\documentclass[12pt]{article} 

\usepackage[utf8]{inputenc} 

\usepackage{geometry} 
\usepackage{graphicx} 

\usepackage{booktabs} 
\usepackage{array} 
\usepackage{paralist} 
\usepackage{verbatim} 
\usepackage{subfig} 
\usepackage{amsmath,amsthm,amssymb,amsfonts}
\usepackage{bbm}
\usepackage{xcolor}
\usepackage{tikz}
\usetikzlibrary{positioning}

\usepackage{fancyhdr} 
\usepackage{sectsty}
\allsectionsfont{\sffamily\mdseries\upshape} 

\usepackage[nottoc,notlof,notlot]{tocbibind} 
\usepackage[titles,subfigure]{tocloft} 

\newtheorem{theorem}{Theorem}[section]
\newtheorem{proposition}[theorem]{Proposition}
\newtheorem{lemma}[theorem]{Lemma}

\newtheorem{remark}[theorem]{Remark}

\title{Random walks on $\mathbb{Z}$ with metastable Gaussian distribution caused by linear drift with application to the contact process on the complete graph}
\author{O.S. Awolude, E. Cator, H. Don}

\begin{document}
\maketitle

\begin{abstract}
	We study random walks on $\mathbb{Z}$ which have a linear (or almost linear) drift towards 0 in a range around 0. This drift leads to a metastable Gaussian distribution centered at zero. We give specific, fast growing, time windows where we can explicitely bound the distance of the distribution of the walk to an appropriate Gaussian. In this way we give a solid theoretical foundation to the notion of metastability. We show that the supercritical contact process on the complete graph has a drift towards its equilibrium point which is locally linear and that our results for random walks apply. This leads to the conclusion that the infected fraction of the population in metastability (when properly scaled) converges in distribution to a Gaussian, uniformly for all times in a fast growing interval.  
\end{abstract}

\section{Introduction}

When an infectious disease is spreading in a population, an important question is how much health care capacity is needed. This directly relates to the question how many people will be infected simultaneously and how this will fluctuate over time. In this paper, we study this question for the contact process (or SIS process) on the complete graph. 

The contact process is a Markovian epidemic model, introduced by Harris \cite{H74} in 1974. Harris considered the process on the lattice $\mathbb{Z}^d$, but it can be defined for any graph. In this model, each individual (represented by a node in the graph) is either infected or healthy. Each infected individual heals at rate 1 and transmits the infection at rate $\tau$ to each of its healthy neighbors. After healing, an individual is again susceptible to infection. For surveys of results on the contact process and related models we refer to \cite{Lig99,DHB13}.

The model has been studied on several other infinite graphs like regular trees \cite{P92,L96,S96} and Galton-Watson trees \cite{HD20,BNNS21}. On infinite graphs, one of the central questions is if the disease will survive forever in the population. Typically, the model exhibits a phase transition in the sense that there exists a (graph-dependent) critical infection rate $\tau_c$ such that the infection goes extinct almost surely for $\tau<\tau_c$ and has strictly positive probability to survive forever for $\tau>\tau_c$. 

On finite graphs, the process almost surely becomes extinct, since the state space is finite and the only absorbing state of the Markov chain is the all-healthy state. Still, a phase transition could be defined by taking graph sequences and considering the extinction time as a function of the number of nodes $n$. For instance, for finite $d$-regular trees, there is a cricital $\tau_c(d)$ such the expected extinction time essentially grows logarithmically with $n$ for $\tau<\tau_c(d)$ and exponentially for $\tau>\tau_c(d)$, see \cite{Stacey01,CMMV14}. Similar results were obtained in \cite{DurLiu88,DurSch88} for intervals $\left\{1,\ldots,n\right\}\in\mathbb{Z}$, with critical value $\tau_c(\mathbb{Z})$ and the extinction time at criticality growing like a polynomial in $n$, see \cite{DurSchTan89}. Extension of these results to higher dimensions can be found in \cite{Mou93}.

If the degrees in the graph grow with $n$, the phases of long and short survival might be separated by a critical value which depends on $n$. For example, $\tau_c = 1/n$ in the complete graph $K_n$ and $\tau_c = n^{-1/2+o(1)}$ in the star graph on $n$ nodes \cite{CM13}. Phase transitions also occur in random graphs, see \cite{BNNS21, CD21} for the Erd\H{o}s-R\'eyni graph and configuration model graphs and \cite{MV16} for random regular graphs.

Many results in the field focus on proving existence of a phase transition, identifying the critical infection rate or finding the expected extinction time or extinction time distribution. In this paper we zoom in on the supercritical \emph{metastable} behaviour. Our intuition is that typically there is at least one nontrivial (i.e. not all healthy) state where healings and infections are balanced. The process has a drift (at least locally) towards this metastable (or \emph{quasistationary}) state. There are some results in the literature where the metastable fraction of infected nodes is identified by solving the balance equations. For instance, in the complete graph $K_n$, with infection rate $\tau =\lambda/n$, the critical point is at $\lambda=1$. If $\lambda>1$ is constant, the metastable solution is at $(1-1/\lambda)\cdot n$ infected nodes and the expected time to extinction is $\exp((\log(\lambda)+\tfrac 1\lambda-1)n+o(n))$ for $n\to\infty$, see \cite{CD21}. A rigorous proof of metastable convergence of the infected fraction to a constant is given in \cite{MouValYao13,CanSch15} for power law random graphs.

In the current paper, considering the contact process on $K_n$, we are not only interested in the metastable fraction of infected nodes, but also in fluctuations. In fact we study a more general problem of sequences of random walks on $\mathbb{Z}$ which have a drift towards 0 which linearly depends on the distance to 0. We show that such walks have a Gaussian metastable distribution, also creating a rigorous framework for defining the notion of Gaussian metastable distributions, using growing time intervals. As an application, we prove (taking limits of time and $n$ in the right way) a limit theorem for the metastable state in the supercritical contact process on the complete graph. Loosely speaking, it states that in metastability the number of infected individuals is distributed as a Gaussian distribution with expectation $(1-1/\lambda)\cdot n$ and variance $n/\lambda$.

\section{Overview of the paper}

The simple symmetric continuous-time random walk $\{W_t,t\geq 0\}$ on $\mathbb{Z}$ is null-recurrent and therefore does not have a proper stationary distribution. Fixing a number $d>0$, the probability $\mathbb{P}(|W_t|\leq d)$ goes to 0 as $t$ goes to infinity. To get positive probabilities, we could scale down the position at time $t$ by a factor $\sqrt{t}$ and look at $W_t/\sqrt{t}$. By the central limit theorem, this scaled position converges to a Gaussian distribution. In this paper, we will see random walks converging to a Gaussian \emph{without} scaling of time.

We look at random walks $W_t$ with a drift towards the origin. If this drift is properly chosen, $\mathbb{P}(|W_t|\leq d)$ will have a limit in $(0,1)$ for all $d>0$, the walk will be positive recurrent and it will have a proper stationary distribution. Our main result (Theorem \ref{thm:main}) essentially says the following. Let $\{W_t,t\geq 0\}$ be a random walk on $\mathbb{Z}$. Suppose $W_t$ has a drift towards the origin which is, in a range around 0, (almost) linear in the distance to 0. Then the \emph{unscaled} position of the walker will converge to a Gaussian, for all times in a growing interval, before the walker will exit the range around $0$ where we control the rates. In this overview, we will explain the spirit of the results a bit further. More precise and more general statements follow in subsequent sections.

Our object of study is a sequence of simple continuous-time random walks on $\mathbb{Z}$.  
For $n\geq 1$, define the random walk $\{W_t^{(n)},t\geq 0\}$ by its transition rates $\lambda_k^{(n)}$ and $\mu_k^{(n)}$, from $k$ to $k+1$ and to $k-1$ respectively. These rates will be chosen so that the walk has ``a linear drift towards 0''. This vague notion will be clarified in Section \ref{sec:bd_processes}, where we discuss some subtleties connected to it. In this introduction we use the term \emph{linear drift} in an informal way. In Section \ref{sec:bd_processes} we also present a few results on birth-death processes which will be intensively used later in the analysis of our random walks. These results have a classical flavour, but we need them to get specific control on the behaviour of our random walks.   

As a first step, we analyze the walks $W_t^{(n)}$ restricted to a finite range.
We let $a_n$ and $\sigma_n$ be suitable increasing functions of $n$ and restrict the state space of $W_t^{(n)}$ to $[-a_n,a_n]\cap\mathbb{Z}$. Inside this state space, we take strictly positive transition rates satisfying (possibly only approximately) 
\begin{equation}\label{eq:rates0}
\frac{\lambda_k^{(n)}}{\mu_{k+1}^{(n)}} = 1-\frac{k}{\sigma_n^2}.
\end{equation}
Note that this relation could be interpreted as a form of linear drift.

Since the walk $W_t^{(n)}$ has a finite state space, it will have a stationary distribution $W_\infty^{(n)}$. Condition \eqref{eq:rates0} guarantees that this distribution is close to a Gaussian with expectation $0$ and variance $\sigma_n^2$. In fact, the ratios $\lambda_k^{(n)}/\mu_{k+1}^{(n)}$ completely determine the stationary distribution. Nevertheless, \eqref{eq:rates0} still leaves a lot of freedom to choose the rates. For instance, for any choice of (positive) rates to the right, the rates to the left can still be chosen to meet the condition. We prove in Section \ref{sec:statdistr} our first main result (Theorem \ref{thm:statdistr}) stating that condition \eqref{eq:rates0} implies
\begin{align}
\frac{W_\infty^{(n)}}{\sigma_n}\stackrel{d}{\longrightarrow} N(0,1).
\end{align}

In Section \ref{sec:mixing}, we fix $n$ and study the convergence of $W_t^{(n)}$ to $W_\infty^{(n)}$ as time $t$ increases. To prove that this convergence is quick, condition $\eqref{eq:rates0}$ is not strong enough. Instead, here it is more natural to consider the difference $\lambda_k^{(n)}-\mu_k^{(n)}$, rather than a ratio of rates. In this context, a linear drift towards 0 means that there exists some constant $d_n>0$ such that 
\begin{align}\label{eq:lindrift}
\lambda_k^{(n)}-\mu_k^{(n)} = -d_nk.
\end{align}
Under such a condition, in Proposition \ref{prop:mixing}, we show that for all $k$
\begin{align}
|\mathbb{P}(W_t^{(n)}\leq k)-\mathbb{P}(W^{(n)}_\infty\leq k)|
\end{align}
decreases exponentially in $t$ with rate $d_n$.

Next, in Section \ref{sec:Zwalks}, we study random walks on $\mathbb{Z}$ which satisfy \eqref{eq:rates0} and \eqref{eq:lindrift} inside $[-a_n,a_n]$. Such a walk does not necessarily have a stationary distribution. If the walk on $\mathbb{Z}$ starts in 0, it will be indistinguishable from the walk restricted to $[-a_n,a_n]$ as long as the boundary of this interval is not yet reached. Note that \eqref{eq:lindrift} implies that 
\begin{align}\label{eq:lindrift1}
\frac{\lambda_k^{(n)}}{\mu_k^{(n)}} \leq 1-\frac{d_nk}{\mu_k^{(n)}} < 1
\end{align}
for $k$ positive. In Proposition \ref{prop:hitboundary}, we prove that this causes the hitting time of $a_n$ to be large with high probability. The hitting time of $-a_n$ can be bounded similarly.

Combining the previous results, we come to our main result for walks on $\mathbb{Z}$, which informally can be stated as follows. Let $W_t^{(n)}$ on $\mathbb{Z}$ have linear drift towards 0 inside $[-a_n,a_n]$ as in \eqref{eq:rates0} and \eqref{eq:lindrift}. Suppose the walk starts in (or close to) 0. Then, for each $n$, there is a large interval $I_n$ in time such that $W_t^{(n)}/\sigma_n$ is in distribution close to a standard Gaussian for all $t\in I_n$. In the limit $n\to\infty$, the difference vanishes. For the precise formulation, we refer to Theorem \ref{thm:main}. We illustrate this theorem by an example and simulation in Section \ref{sec:Zwalks}.

Finally, in Section \ref{sec:contactprocess}, we apply our results to the contact process on the complete graph $K_n$. Each node in the graph represents an individual which can either be healthy or infected. If a node is infected, it is infecting its healthy neighbors at rate $\tau$ and it is healing at rate 1. Since the graph is complete, only the size of the infected set is important for the evolution of the process. Given this size, it is irrelevant which individuals are infected. The total number of infected $X_t$ behaves like a simple random walk on $\{0,1,\ldots,n\}$. The stationary distribution is trivial: all nodes are healthy. However, if the infection rate is of the form $\tau = \tfrac{\lambda}{n}$ for some $\lambda>1$, then the healings and infections are balanced if the infected set has size $\tfrac{\lambda-1}{\lambda}\cdot n$. The process $X_t$ has a drift towards this equilibrium, which is locally approximately linear as in \eqref{eq:rates0} and \eqref{eq:lindrift}. If $X_0$ is close to $\tfrac{\lambda-1}{\lambda}\cdot n$, our result for random walks on $\mathbb{Z}$ can be applied. We did some extra work to weaken the initial condition. This leads to Theorem \ref{thm:contactprocess}, our main result for the contact process on $K_n$:  there exists a constant $c_1>0$ such that if $X_0 \geq c_1\log(n)$, then $X_t$ is in distribution close to a Gaussian with expectation $\tfrac{\lambda-1}{\lambda}\cdot n$ and variance $\tfrac{n}{\lambda}$ for all $t\in [n,e^{n\cdot g(n)}]$, where $g(n)$ goes to zero arbitrarily slowly.

Studying the metastable distribution and behaviour of the contact process on the complete graph has been the inspiration for this paper. We needed to control the contact process both for growing $t$ and for growing $n$, at the same time. This forced us to study classical objects like simple random walks. Many results in this paper might therefore feel familiar, especially those in Section 2, 3 and 4, and although references to the specific results we needed are not easily found, we do not claim that they are all new; we did add their proofs to be self-contained and complete. Sections 5 and 6 contain the most original results of this paper.

\section{Extinction of birth-death processes}\label{sec:bd_processes}

Before heading to our main results, this section is devoted to some preliminaries. In Section \ref{sec:drift} we elaborate on the notion of drift.  Next, we derive a few results on the extinction of birth-death processes. All processes in this paper are simple random walks on finite or infinite intervals in $\mathbb{Z}$, either in discrete or in continuous time. By a birth-death process we mean the special case where the state space is $\mathbb{N}$ or $[n]:=\{0,1,\ldots,n\}$ and where 0 is an absorbing state. 

In Section \ref{sec:discretedrift} we analyze the extinction time of a discrete-time birth-death process with positive drift. After that, in \ref{sec:poslindrift} and \ref{sec:neglindrift}, we study extinction in the continuous-time process. Usually, one is interested in the expected extinction time, but we aim for tail bounds for the extinction time distribution. We keep our exposition self-contained, since we could not easily find these results in the literature. We will consider both positive and negative drift, corresponding to fast and slow extinction respectively.

\subsection{The notion of drift}\label{sec:drift}

For discrete-time random walks, by a positive drift in state $k$ we just mean that the probability to go right is greater than the probability to go left. For continuous-time random walks, the situation is more subtle and the notion of a linear drift towards zero could be defined in different ways. Here we discuss three possible definitions: linearity of $\lambda_k/\mu_{k+1}$ as in \eqref{eq:rates0}, linearity of $\lambda_k/\mu_{k}$, and linearity of the differences $\lambda_k-\mu_k$ as in \eqref{eq:lindrift}. Each of these definitions has different implications for the random walk, which we illustrate by an example. 
 
Consider a simple random walk $\{W_t\}_{t\geq 0}$ on state space $[n]$ with rates to the right and left respectively
\begin{align}
&\lambda_k = (1+\varepsilon)^{k^2}, &&k=0,\ldots, n-1 \label{eq:examplelam}\\
&\mu_k = (1+\varepsilon)^{k(k-1)}, &&k = 1,\ldots,n,\label{eq:examplemu} 
\end{align}
where $\varepsilon$ is some small positive number. Observe that 
\begin{align}\label{eq:longdrift}
\frac{\lambda_k}{\mu_{k+1}} = (1+\varepsilon)^{-k} \approx 1-k\varepsilon,\qquad k=0,\ldots,n-1,
\end{align}
so that this random walk has an approximately linear drift towards 0. As we will see in Section \ref{sec:statdistr}, this type of drift leads to a stationary distribution which is close to a Gaussian with expectation 0 and variance $\sigma^2 = \varepsilon^{-1}$. In the current example, this Gaussian is truncated at $0$ and $n$. For $n = 100$ and $\varepsilon = 1/1000$, we computed the stationary distribution of $W_t$. In Figure \ref{fig:example}, left plot, a histogram of this distribution is compared to the corresponding truncated Gaussian. The highest probabilities correspond to positions close to 0.

On the other hand, in each state $k$, the rate to the right is higher than the rate to the left. This suggests a linear drift away from 0:
\begin{align}\label{eq:shortdrift}
\frac{\lambda_k}{\mu_k} = (1+\varepsilon)^k \approx 1+k\varepsilon.
\end{align}
Let $\{D_i\}_{i=0}^\infty$ be the jump chain corresponding to $W_t$, i.e. if $T_i$ is the time of the $i$th step of $W_t$, we let 
\begin{align}
D_0 = W_0, \qquad D_i = W_{T_i}, i\geq 1.
\end{align}
$D_i$ is a discrete-time Markov process with transition probabilities proportional to the transition rates of $W_t$. If $p_{k,l}$ is the probability from $k$ to $l$, then 
\begin{align}
p_{k,k+1} = \frac{\lambda_k}{\lambda_k+\mu_k},\qquad p_{k,k-1} = \frac{\mu_k}{\lambda_k+\mu_k},\qquad k = 1,\ldots,n-1.
\end{align} 
Since $D_i$ is periodic, it does not have a stationary distribution. We modify $W_t$ to have a positive rate from $0$ to $0$ and from $n$ to $n$. This does not actually change the stationary distribution of $W_t$, but it allows to assume that 
\begin{align}
p_{0,0} = p_{0,1} = p_{n,n-1}=p_{n,n} = \tfrac12,
\end{align}
making $D_i$ aperiodic. The stationary distribution of $D_i$ is plotted in Figure \ref{fig:example}, right plot. As opposed to $W_t$ itself, the corresponding jump chain $D_i$ is atracted to $n$ rather than to 0. This means that $W_t$ spends much more time close to 0, but the states close to $n$ are visited much more frequently. This is possible because $W_t$ is moving back and forth much faster when close to $n$; the rates in \eqref{eq:examplelam} and \eqref{eq:examplemu} rapidly increase with $k$. 

The example illustrates that the roles of the ratios in \eqref{eq:longdrift} and \eqref{eq:shortdrift} is completely different. One could say that the ratio in \eqref{eq:longdrift} is decisive for the long term behaviour, while the ratio in \eqref{eq:shortdrift} determines the short term behaviour. Indeed, 
\begin{align}
\frac{\lambda_k}{\mu_{k+1}} &= \lim_{t\to\infty}\frac{\mathbb{P}(W_t=k+1)}{\mathbb{P}(W_t=k)},\qquad \frac{\lambda_k}{\mu_k} = \lim_{t\to 0}\frac{\mathbb{P}(W_t=k+1\mid W_0=k)}{\mathbb{P}(W_t=k-1\mid W_0=k)}.
\end{align}
In this paper, we therefore speak about \emph{long term drift} and \emph{short term drift} respectively. We will slightly abuse terminology by calling these drifts positive if the corresponding ratios exceed 1.
 
\begin{figure} 
\includegraphics[width=\linewidth]{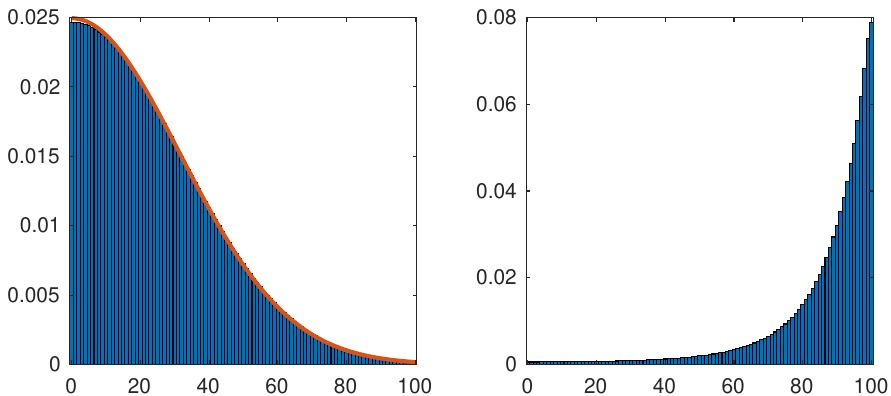}\caption{Stationary distributions of the continuous-time walk $W_t$ (left) and the corresponding jump chain $D_i$ (right).}\label{fig:example}
\end{figure}

Clearly the rates can still be scaled without changing any of the ratios. This is basically a scaling of time which speeds up or slows down the whole process. Such a scaling does influence the difference $\lambda_k-\mu_k$, which is yet another quantity related to the concept of drift. We will call this the \emph{speed} of the process in state $k$, motivated by Lemma \ref{lem:mixing_differential}. The direction of the speed is equal to the direction of the short term drift. 

For our main theorem for random walks on $\mathbb{Z}$ (Theorem \ref{thm:main}) all three notions of drift will play a role. 

\subsection{Discrete-time process with positive drift}\label{sec:discretedrift}

Consider a discrete-time birth-death process $\{X_i,i\in \mathbb{N}\}$ on $[n]$ with constant birth probability $p$ and death probability $q=1-p$. In state $n$, no birth is possible, so the death probability is $1$. In state $0$, births are impossible so that state 0 is an absorbing state, see the diagram below.

\begin{center}
	\begin{tikzpicture}[-latex  ,node distance =2.4 cm and 2.4cm ,on grid ,
	semithick ,
	state/.style ={ circle ,top color =white , bottom color = white!20 ,
		draw, black , text=black , minimum width =1.2 cm,inner sep = 1pt}]
	\node[state] (1) {$0$};
	\node[state] (2) [right =of 1] {$1$};
	\node[state] (3) [right =of 2] {$2$};
	\node[state] (4) [right  =of 3] {$n-2$};
	\node[state] (5) [right  =of 4] {$n-1$};
	\node[state] (6) [right  =of 5] {$n$};
	\draw[->] (1) edge[out=200,in=160,looseness=7] node[left,inner sep=4pt] {$1$} (1);
	\path (2) edge [bend left =20] node[below] {$q$} (1);
	\path (2) edge [bend left =20] node[above] {$p$} (3);
	\path (3) edge [bend left =20] node[below] {$q$} (2);
	\path (3) edge [dotted, bend left =20] node[below] {} (4);
	\path (4) edge [dotted, bend left =20] node[below] {} (3);
	\path (4) edge [bend left =20] node[above] {$p$} (5);
	\path (5) edge [bend left =20] node[below] {$q$} (4);
	\path (5) edge [bend left =20] node[above] {$p$} (6);
	\path (6) edge [bend left =20] node[below] {$1$} (5);
	\end{tikzpicture}
\end{center}  

We will be interested in tail probabilities of the extinction times if the process starts in state $k\in [n]$ for the case $p>q$. There is a drift to the right, so we expect the process to survive for a long time. Our aim is to prove that indeed the probability on quick extinction is very small. The extinction times are given by
\begin{align}
T_k := \inf\{i:X_i=0\mid X_0=k\}.
\end{align}
The next proposition proves that the extinction time $T_n$ is exponential in $n$.

\begin{proposition}\label{prop:extinctiondiscrete}
Let $\{X_i,i\in \mathbb{N}\}$ be a birth-death process on $[n]$ with birth probability $p>\frac12$ and death probability $q=1-p$. Let $X_0=n$. For any $i\geq 0$, the extinction time $T_n$ satisfies
\begin{align}
\mathbb{P}(T_n< i) \leq (i-1)\cdot\left(\frac q p\right)^{n-1}.
\end{align}
\end{proposition}
\begin{proof}
Define $\pi_k$ as the probability of leaving the interval $(0,k)$ at $0$, when starting at $k-1$. Then $\pi_1=1$ and for $2\leq k\leq n$
\begin{align} \pi_k = q\cdot(\pi_{k-1} + (1-\pi_{k-1})\pi_k).\end{align}
This implies that
\begin{align}
\pi_k = \frac{q\pi_{k-1}}{p+q\pi_{k-1}}\leq \frac{q}{p}\pi_{k-1}\leq \left(\frac{q}{p}\right)^{k-1}.
\end{align}
Each time we are at position $n$, we will move to $n-1$ and we will have probability $\pi_n$ to reach $0$ before returning to $n$, and otherwise we end up at $n$ again. The number of attempts to reach $0$ has a geometric distribution with parameter $\pi_n$ and the number of steps we take to reach $0$ is certainly at least the number of these attempts. Therefore,
\begin{align} \mathbb{P}(T_n<i) \leq 1- (1-\pi_n)^{i-1} \leq (i-1)\left(\frac{q}{p}\right)^{n-1}.\end{align}
\end{proof}
Using a straightforward coupling, the same result is valid when all birth probabilities are at least $p$. Note that $\lim_{n\to\infty}\mathbb{P}(T_n<i)= 0$ whenever $i = o((p/q)^n)$. This result is reasonably sharp; one can show that the excursions from $n$ to $n$ have finite expected length (close to $(1-(q/p))^{-1}$) and that $\mathbb{E}[T_n]$ is of order $(p/q)^n$. 


\subsection{Continuous-time process with positive speed}\label{sec:poslindrift}

We now consider the case of a continuous-time birth-death process $\{Y_t,t\geq 0\}$ on $[n]$ with birth rate $\lambda_k$ and death rate $\mu_k$ in state $k$ (which from now on will be standard notation). By definition, $\lambda_0=\mu_0=\lambda_n = 0$. 
Let $p_k$ be the probability on a birth in state $k\in\{1,\ldots,n-1\}$ and $q_k$ the probability on a death. Then
\begin{align}
p_k = \frac{\lambda_k}{\lambda_k+\mu_k},\qquad q_k = \frac{\mu_k}{\lambda_k+\mu_k}.
\end{align}
Now suppose $Y_0=n$. 
We define the extinction time by
\begin{align}
T_n = \inf\{t\geq 0:Y_t=0\}.
\end{align}
The extinction time is finite almost surely. We assume the process to be right-continuous, so that $Y_{T_n}=0$ if $T_n$ is finite. 

The next proposition essentially says that the extinction time grows like the $n$th power of the short term drift (for constant rates). Since $\kappa_n$ below is an upper bound for the absolute speed, $\kappa_nt$ is a scaling of time which adjusts for the speed of the process. In analogy to the discrete-time case, if all birth rates are $\lambda$ and all death rates are $\mu$, then $\lim_{n\to\infty}\mathbb{P}(T_n\leq t) = 0$ if $\kappa_n t = o((\lambda/\mu)^n)$. 

\begin{proposition}\label{prop:extinctioncontinuous} Let $\{Y_t,t\in\mathbb{R}^+\}$ be a birth-death process on $[n]$ with birth rates $\lambda_k^{(n)}$, death rates $\mu_k^{(n)} <\lambda_k^{(n)}$ and $Y_0=n$. Let $\kappa_n = \max_k\{\lambda_k^{(n)}+\mu_k^{(n)}\}$ and $\rho_n = \max_k\{\mu_k^{(n)}/\lambda_k^{(n)}\}$. For any $\varepsilon>0$, there exists $C_1,C_2>0$ such that for all $n\geq 1$ and for all $t>0$ with $\kappa_n t\geq C_1 + C_2n\log(\rho_n^{-1})$,
\begin{align}
	\mathbb{P}(T_n< t)\leq (1+\varepsilon)\kappa_nt\rho_n^{n-1}.
\end{align}
\end{proposition}

\begin{proof}
 Define $S_i$ to be the $i$th event time. If $i$ exceeds the total number of events, we set $S_i=\infty$. Now the process  $\{X_i,i\in\mathbb{N}\}$ defined by
\begin{align}
X_i = \begin{cases}Y_{S_i}\qquad\qquad&S_i\leq T_n,\\
0&S_i=\infty,\end{cases}
\end{align}
is a discrete-time birth-death process on $[n]$ with transition probabilities $p_k$ and $q_k$. Choose some $i\in\mathbb{N}$ and $t>0$ and note that the extinction time $T_n$ satisfies
\begin{align}
\mathbb{P}(T_n< t) &= \mathbb{P}(S_i< t, T_n< t)+\mathbb{P}(S_i> t, T_n< t)\nonumber\\
&\leq \mathbb{P}(S_i< t)+\mathbb{P}(S_i=\infty)\nonumber\\& \leq \mathbb{P}(S_i< t)+\mathbb{P}(X_i=0).\label{eq:extinctionprob}
\end{align}
We proceed to bound $\mathbb{P}(S_i< t)$. The $i$th event time is a sum of exponential waiting times with rate at most $\kappa_n$. Therefore, if $E_1,\ldots E_i$ are independent exponential random variables with rate $\kappa_n$, then 
\begin{align}
\mathbb{P}(S_i< t) \leq \mathbb{P}\left(\sum_{j=1}^i E_j< t\right) = \mathbb{P}\left(e^{-a\sum E_j}> e^{-at}\right) \leq \frac{\mathbb{E}\left[e^{-a\sum E_j}\right]}{e^{-at}},
\end{align}
for $a>0$ by Markov's inequality. Using independence of the $E_j$,
\begin{align}
\mathbb{P}(S_i< t) \leq \left(\frac{\kappa_n}{a+\kappa_n}\right)^i\cdot e^{at}.
\end{align}
Finally, optimizing over $a$ gives $a = \frac i t-\kappa_n$, so that
\begin{align}
\mathbb{P}(S_i< t) \leq \left(\frac{\kappa_n t}{i}\right)^i\cdot e^{i-\kappa_n t}.
\end{align}
 The bound is trivial for $i\leq \kappa_n t$, but is less than 1 and decreasing for $i>\kappa_n t$. 
%
First assume that we would have $\kappa_nt\geq2+\frac{4}{\varepsilon}$. Then we can choose $i$ as follows:
 \begin{align}
 \kappa_nt+1\leq \frac{1+\tfrac{3\varepsilon}{4}}{1+\tfrac{\varepsilon}{2}}\kappa_nt\leq i = \left\lfloor \frac{1+\varepsilon}{1+\tfrac{\varepsilon}{2}}\kappa_nt\right\rfloor \leq \frac{1+\varepsilon}{1+\tfrac{\varepsilon}{2}}\kappa_nt.
 \end{align}
Choose $\delta$ such that $i = e^{\delta}\kappa_nt$, so 
  \begin{align}
  \log\left(\frac{1+\frac{3\varepsilon}{4}}{1+\frac{\varepsilon}{2}}\right) \leq \delta \leq \log\left(\frac{1+\varepsilon}{1+\frac{\varepsilon}{2}}\right),
  \end{align}
  and call the lower bound $\tilde \delta$. Define $\eta = e^{-\tilde\delta} - (1-\tilde\delta)>0$. Then
  \begin{align}
  \mathbb{P}(S_i< t) &\leq \exp(-\delta i +i-\kappa_n t) =\exp{((1-\delta - e^{-\delta})e^\delta\kappa_n t)} \leq \exp(-\eta e^\delta\kappa_nt).
  \end{align}
  Define $C_1$ and $C_2$, which only depend on $\varepsilon$, to be
  \begin{align} 
  C_1 = \max\left\{\frac{1}{\eta}\log\left(\frac{2}{\varepsilon}\right),2+\frac{4}{\varepsilon}\right\}\ \mbox{and}\ C_2 = \frac{1}{\eta}.
  \end{align}
  Then for $\kappa_nt\geq C_1 + C_2 n \log(\rho_n^{-1})$, choose $i$ and $\delta$ as above. We conclude that
  \begin{align}
  \mathbb{P}(S_i< t) &\leq \exp(-\eta  \kappa_nt) \leq \frac12 \varepsilon \rho_n^n \leq \frac12 \varepsilon e^\delta \kappa_n t\rho_n^{n-1}.
  \end{align}
  The other term in (\ref{eq:extinctionprob}) can be bounded using Proposition \ref{prop:extinctiondiscrete} and noting that $q_k/p_k\leq \rho_n$ for all $k$: 
  \begin{align}
  \mathbb{P}(X_i=0) \leq (i-1)\rho_n^{n-1} \leq e^{\delta}\kappa_n t \rho_n^{n-1}. 
  \end{align}
  Adding the two terms, we get
  \begin{align}
  \mathbb{P}(T_n< t) \leq (1+\varepsilon)\kappa_n t \rho_n^{n-1},
  \end{align}
  completing the proof.
\end{proof}

%

\subsection{Continuous-time process with negative linear speed}\label{sec:neglindrift}

Here we again consider a continuous-time birth death process $\left\{Y_t,t\geq 0\right\}$ on $[n]$ with rates $\lambda_k$ and $\mu_k$. The next lemma is standard and therefore presented without proof. It gives a differential equation for $\mathbb{E}[Y_t]$ in terms of the rates. It demonstrates that $\lambda_k-\mu_k$ is a measure of the expected speed of a walker in state $k$. 

\begin{lemma}\label{lem:mixing_differential}
	Let $\{Y_t,t\in\mathbb{R}^+\}$ be a birth-death process on $[n]$, with arbitrary initial distribution $Y_0$. The expected position $\mathbb{E}[Y_t]$ at time $t>0$ satisfies
	\begin{align}
	\frac{d}{dt}\mathbb{E}[Y_t] = \mathbb{E}[\lambda_{Y_t}-\mu_{Y_t}]
	\end{align}
\end{lemma}

%
%

We are now interested in the case where the speed $\lambda_k-\mu_k$ is negative and linear in $k$. The speed is constant over time, but may depend on $n$, for instance $\lambda_k-\mu_k=-k/n$. So we actually consider a sequence of processes, indexed by $n$. The following lemma then gives asymptotically ($n\to\infty$) almost sure upper bounds on the extinction time, which depend on the speed. For fixed $n$, a linear speed towards 0 guarantees exponentially fast extinction.

\begin{lemma}\label{lem:linear_drift} Suppose there is some constant $d = d(n) >0$ such that the transition rates of the birth death process $\{Y_t\}$ satisfy $\lambda_k-\mu_k = -d\cdot k$ for all $0\leq k\leq n$. Take any initial distribution $Y_0$. Then
	\begin{align}
\mathbb{P}(T_n > t) \leq n\cdot e^{-d\cdot t}.
	\end{align}

\end{lemma}
\begin{proof}
	First note that by a coupling argument we can assume that $Y_0 \equiv n$. Now Lemma \ref{lem:mixing_differential} gives us 
	\begin{align}
	\frac{d}{dt}\mathbb{E}[Y_t] = -d\cdot \mathbb{E}[Y_t],
	\end{align}
	with solution $\mathbb{E}[Y_t] = \mathbb{E}[Y_0]\cdot e^{-d\cdot t} = n\cdot e^{-d\cdot t}$.
	Markov's inequality gives 
	\begin{align}
	\mathbb{P}(T_n > t) = \mathbb{P}(Y_t\geq 1)\leq \mathbb{E}[Y_t] = n\cdot e^{-d\cdot t}.
	\end{align}
	\end{proof}

\section{A random walk with Gaussian stationary distribution}\label{sec:statdistr}

In this section, we study a continuous-time random walk $\{W_t^{(n)},t\in\mathbb{R}^+\}$ on $\mathbb{Z}\cap [-b_n,b_n]$ which depends on $n$ and has transition rates
\begin{eqnarray}
k\rightarrow k+1 & \text{with rate} & \lambda_k^{(n)} \qquad\text{for}\qquad -b_n\leq k \leq b_n-1,\\
k\rightarrow k-1 & \text{with rate} & \mu_k^{(n)} \qquad\text{for}\qquad b_n+1\leq k \leq b_n.
\end{eqnarray}
It will turn out that the long term drift determines the stationary distribution. Let $\sigma_n$ be an increasing positive function of $n$ with $\sigma_n\to\infty$ and take $b_n$ and $a_n$ such that 
\begin{equation}\label{eq:defbn}
\left\lceil\sigma_n\sqrt{2\log(\sigma_n)}\right\rceil \leq a_n \leq b_n\qquad\text{and}\qquad a_n/\sigma_n^2\to 0.
\end{equation}
Assume that for each $n$ there is a function $\delta_n$ taking small values and $\eta_n\geq 0$ such that the rates statisfy
\begin{equation}\label{eq:rates1}
\log\left(\frac{\lambda_k^{(n)}}{\mu_{k+1}^{(n)}}\right) \left\{ \begin{array}{ll} \geq \eta_n & \text{if}\quad -b_n\leq k <- a_n\\
= -\frac{k}{\sigma_n^2} + \delta_n(k)& \text{if}\quad -a_n\leq k \leq a_n\\
\leq -\eta_n & \text{if}\quad a_n<k \leq b_n.
\end{array}\right.
\end{equation}
In the interval $[-a_n,a_n]$, this walk has a long term drift towards the origin $0$, which is approximately linear in $k$:
\begin{align}
\frac{\lambda_k^{(n)}}{\mu_{k+1}^{(n)}} \approx 1-\frac{k}{\sigma_n^2}+\delta_n(k).
\end{align}
We introduce the error term $\delta_n(k)$ to allow the ($\log$ of the) drift to be not exactly linear in $k$. When very far away ($|k|>a_n$), there is simply a constant drift towards the origin or no drift at all. 

We will prove that with a proper scaling in $n$ the stationary distribution of this random walk converges to a Gaussian as $n\to\infty$ under suitable conditions on the parameters. There will be \emph{no} scaling in $t$. So for fixed $n$, we can take the limit $t\to\infty$ and $W_t^{(n)}$ will converge in distribution.

Condition \eqref{eq:rates1} is not yet sufficient to conclude that the random walk quickly reaches its stationary distribution. For this we will need more control on the rates, see also the discussion in Section \ref{sec:drift}. This issue will be settled in Section \ref{sec:mixing}. 

\subsection{Derivation of the stationary distribution}

The walk $W_t^{(n)}$ has a finite state space and is irreducible. It therefore has a unique stationary distribution $\pi_n(k)$.
\begin{proposition}\label{prop:statdist} For each $n$, let $b_n\geq a_n$ be integers satisfying $\eqref{eq:defbn}$ and take rates as in (\ref{eq:rates1}). 
Let $K$ be a constant and assume that 
\begin{align}\label{eq:boundondelta}
	\sum_{k=-a_n}^{a_n}|\delta_n(k)|\leq K\cdot\frac{a_n}{\sigma_n^2}.
\end{align}
Furthermore, assume that
\begin{equation}\label{eq:boundoneta}
\eta_n \geq \sqrt{\frac{2}{\pi}}\frac{\sigma_n}{a_n}e^{-\frac12 a_n^2/\sigma_n^2}\ \ \text{or}\ \ b_n-a_n \leq \sqrt{\frac{\pi}{2}} \frac{a_n}{\sigma_n}e^{\frac12 a_n^2/\sigma_n^2}.
\end{equation}
	Then for $n$ large enough
	\begin{align}\label{eq:piresult}
	\pi_n(k) = \frac{1}{\sqrt{2\pi}\sigma_n}\exp\left\{-\frac12 \frac{k^2}{\sigma_n^2}+\varepsilon_{n,k}\right\},\qquad  -a_n\leq k\leq a_n,
	\end{align}
	with $|\varepsilon_{n,k}| \leq (4+2K)a_n/\sigma_n^2$, and
	\begin{align}
	\sum_{\{k\ :\ a_n<|k|\leq b_n\}} \pi_n(k) \leq \frac{a_n}{\sigma_n^2}e^{(3+2K)a_n/\sigma_n^2}. 
	\end{align}
\end{proposition}

\begin{remark}
Suppose the rates satisfy
\begin{equation}\label{eq:rates2}
\frac{\lambda_k^{(n)}}{\mu_{k+1}^{(n)}} \left\{ \begin{array}{ll} \geq \exp(\eta_n) & \text{if}\quad -b_n\leq k <- a_n\\
= 1-\frac{k}{\sigma_n^2}+\delta_n(k)& \text{if}\quad -a_n\leq k \leq a_n\\
\leq \exp(-\eta_n) & \text{if}\quad a_n<k \leq b_n,
\end{array}\right.
\end{equation}
with $\sum_{k=-a_n}^{a_n}|\delta_n(k)|\leq K\cdot a_n/\sigma_n^2$. Then 
\begin{align}
\sum_{k=-a_n}^{a_n} \left|\log\left(\frac{\lambda_k^{(n)}}{\mu_{k+1}^{(n)}}\right)+\frac{k}{\sigma_n^2}\right| &=\sum_{k=-a_n}^{a_n} \left|\delta_n(k)+O\left(\left(\frac{k}{\sigma_n^2}+\delta_n(k)\right)^2\right)\right|\\
&\leq K\cdot\frac{a_n}{\sigma_n^2}+O\left(\frac{a_n^3}{\sigma_n^4}\right) = (K+o(1))\cdot\frac{a_n}{\sigma_n^2}, 
\end{align}
so that \eqref{eq:rates2} is an alternative to \eqref{eq:rates1} under which Proposition \ref{prop:statdist} holds as well.
\end{remark}

\begin{remark}
The error term $\varepsilon_{n,k}$ in \eqref{eq:piresult} becomes larger if $a_n$ is increased. This might be unexpected, since we have more control on the rates in \eqref{eq:rates1}. However, when $a_n$ increases, at the same time the assumptions in \eqref{eq:boundondelta} and \eqref{eq:boundoneta} are weakened, which explains why the error term increases. The choice of $a_n$ can be adapted to the situation in which one wants to apply the results, but should satifsy \eqref{eq:defbn}.
\end{remark}

\begin{proof}[Proof of Proposition \ref{prop:statdist}]
	Fix $n\geq 1$. Then for all $k = -b_n,\ldots,b_n-1$, the stationary distribution satisfies
	\begin{align}\label{eq:piratio0}
	\frac{\pi_n(k+1)}{\pi_n(k)} = \frac{\lambda_k^{(n)}}{\mu_{k+1}^{(n)}}.
	\end{align}
	Iterating gives for all $k = 1,\ldots,a_n$,
	\begin{align}
	\frac{\pi_n(k)}{\pi_n(0)} &= \prod_{i=0}^{k-1} \frac{\lambda_i^{(n)}}{\mu_{i+1}^{(n)}}  =\exp\left\{-\sum_{i=0}^{k-1} \frac{i}{\sigma_n^2}+\sum_{i=0}^{k-1}\delta_n(i)\right\}\\
	\frac{\pi_n(0)}{\pi_n(-k)} &= \prod_{i=1}^{k} \frac{\lambda_{-i}^{(n)}}{\mu_{-i+1}^{(n)}}  =\exp\left\{\sum_{i=1}^{k} \frac{i}{\sigma_n^2}+\sum_{i=1}^{k}\delta_n(-i)\right\}.
	\label{eq:piratio}
	\end{align}
	This implies that for all $-a_n\leq k\leq a_n$ 
	\begin{equation}\label{eq:piineq}
	\pi_n(0)\exp\left\{-\frac12\frac{k^2}{\sigma_n^2}\right\}e^{-(K+1)a_n/\sigma_n^2} \leq \pi_n(k)\leq \pi_n(0)\exp\left\{-\frac12\frac{k^2}{\sigma_n^2}\right\}e^{(K+1)a_n/\sigma_n^2}.
	\end{equation}
	Note that it is hard to improve on these inequalities, since they are also caused by the discrete nature of the probability distribution $\pi_n$. Lemma \ref{lem:sumnormdens} (see Appendix) states that for $a\geq \sigma\sqrt{2\log(\sigma)}$ and $\sigma$ large enough,
	\begin{align}
	\sqrt{2\pi}\sigma e^{-a/\sigma^2} \leq \sum_{k=-a}^{a} e^{-\frac12\frac{k^2}{\sigma^2}} \leq \sqrt{2\pi}\sigma e^{a/\sigma^2}.
	\end{align}
	We use this to obtain
	\begin{align}\label{eq:pisum}
	 \pi_n(0)\sqrt{2\pi}\sigma_n e^{-(K+2)a_n/\sigma_n^2}\leq \sum_{k = -a_n}^{a_n} \pi_n(k) \leq \pi_n(0)\sqrt{2\pi}\sigma_n e^{(K+2)a_n/\sigma_n^2}.
	\end{align}
	For $a_n<k\leq b_n$, by \eqref{eq:rates1} and  \eqref{eq:piratio0} we get
	\begin{align}\label{eq:piineqfar}
	\pi_n(k)\leq \pi_n(a_n)e^{-\eta_n (k-a_n)}\ \ \text{and}\ \ \pi_n(-k)\leq  \pi_n(-a_n)e^{-\eta_n (k-a_n)}.
	\end{align}
	This implies that (using \eqref{eq:piineq} and then \eqref{eq:boundoneta}),
	\begin{align}\label{eq:pisumfar}
	\sum_{a_n+1\leq |k| \leq b_n}\pi_n(k) &\leq (\pi_n(a_n)+\pi_n(-a_n))\sum_{k=1}^{b_n-a_n}e^{-\eta_nk}\nonumber\\
	&\leq 2\pi_n(0)e^{-\frac12 \frac{a_n^2}{\sigma_n^2}}e^{(K+1)a_n/\sigma_n^2} \cdot \min\{b_n-a_n,\int_0^\infty e^{-\eta_ns}ds\}\nonumber\\
	&\leq \pi_n(0)e^{(K+1)a_n/\sigma_n^2} \cdot \frac{\sqrt{2\pi}a_n}{\sigma_n}.
	\end{align}
	Now we will bound $\pi_n(0)$ by using that 
	\begin{align}\label{eq:pisum1}
	\sum_{k=-a_n}^{a_n} \pi_n(k) \leq \sum_{k=-b_n}^{b_n} \pi_n(k) = 1.
	\end{align}
	Combining with the left hand side of \eqref{eq:pisum} gives the upper bound 
	\begin{align}
	\pi_n(0) \leq \frac{1}{\sqrt{2\pi}\sigma_n} e^{(K+2)a_n/\sigma_n^2}.
	\end{align}
	The right hand side of \eqref{eq:pisum} together with \eqref{eq:pisumfar} and \eqref{eq:pisum1} leads to
	\begin{align}
	1 \leq \pi_n(0)\sqrt{2\pi}\sigma_n e^{(K+2)a_n/\sigma_n^2}(1+\frac{a_n}{\sigma_n^2})\leq \pi_n(0)\sqrt{2\pi}\sigma_n e^{(K+3)a_n/\sigma_n^2},
	\end{align}
	which gives the lower bound
	\begin{align}
	\pi_n(0) \geq \frac{1}{\sqrt{2\pi}\sigma_n} e^{-(K+3)a_n/\sigma_n^2}.
	\end{align}
	Using \eqref{eq:piineq} for the last time gives the final result \eqref{eq:piresult} for all $-a_n\leq k\leq a_n$. Note that we have also proved that
	\[ \sum_{\{k\ :\  a_n<|k|\leq b_n\}} \pi_n(k) \leq \frac{a_n}{\sigma_n^2} e^{(2K+3)a_n/\sigma_n^2}.\]
\end{proof}

\begin{theorem}\label{thm:statdistr}
Let $W_\infty^{(n)}$ be a random variable corresponding to the stationary distribution of the random walk $W_t^{(n)}$, i.e. $\mathbb{P}(W_\infty^{(n)}=k)=\pi_n(k)$. Then, under the conditions of Proposition \ref{prop:statdist} and as $n\to\infty$,
\begin{align}
\frac{W_\infty^{(n)}}{\sigma_n} \stackrel{d}{\longrightarrow} Z,
\end{align}
where $Z$ is a standard Gaussian.
\end{theorem}

\begin{proof}
	This is a direct consequence of Proposition \ref{prop:statdist}. Fix $x\in \mathbb{R}$. For $n$ big enough we have that $|x|\sigma_n\leq a_n$. Therefore
	\begin{align}
	\sum_{k=-a_n}^{\lfloor x\sigma_n\rfloor}\pi_n(k)\leq \mathbb{P}(W_\infty^{(n)}\leq x\sigma_n) \leq \sum_{k=-a_n}^{\lfloor x\sigma_n\rfloor}\pi_n(k) +\sum_{a_n<|k|\leq b_n}\pi_n(k). 
	\end{align}
	Since $a_n/\sigma_n^2\to 0$, $\sigma_n\to +\infty$ and $a_n^2/\sigma_n^2\to +\infty$, we have that
	\begin{align}
	\sum_{k=-a_n}^{\lfloor x\sigma_n\rfloor}\pi_n(k) = \frac{1+o(1)}{\sqrt{2\pi}\sigma_n}\sum_{k=-a_n}^{\lfloor x\sigma_n\rfloor}e^{-\frac12k^2/\sigma_n^2}\to\mathbb{P}(Z\leq x),\qquad \sum_{a_n<|k|\leq b_n}\pi_n(k)\to 0,
	\end{align}
	proving that $\mathbb{P}(W_\infty^{(n)}/\sigma_n\leq x) \to \mathbb{P}(Z\leq x).$
\end{proof}

\section{Mixing of the random walk}\label{sec:mixing}

In this section, our aim is to give conditions on the rates under which the walk $W_t^{(n)}$ restricted to $\{-a_n,\ldots,a_n\}$ (i.e. with $a_n=b_n$) converges quickly to its stationary distribution. The idea will be to look at the distance between the two independent walkers and to show that they will meet quickly. In this setting, the speed of the walkers is the relevant measure of drift. The difference between the walkers will again be a random walk, which is restricted to the positive integers. In the remainder of this section, we omit the dependence on $n$ and simply write $W_t$ for $W_t^{(n)}$, and $\mu_k$ and $\lambda_k$ for the rates.

 Consider two independent realizations of $W_t$ in different positions $k<l$. The distance between the walkers then increases at rate $\lambda_{l}+\mu_{k}$ and decreases at rate $\lambda_{k}+\mu_{l}$. This means that the speed of the difference is given by
\begin{align}\label{eq:driftdiff}
(\lambda_{l}+\mu_{k})-(\lambda_{k}+\mu_{l}) = (\lambda_{l}-\mu_{l})-(\lambda_{k}-\mu_{k}), \end{align}
which is equal to the difference of the speeds of the two walkers. Note that the signs of the individual speeds are not that relevant. For instance, the walker on the right could have a positive speed and still meet quickly if the walker on the left has a higher positive speed. In particular, if the speed difference in \eqref{eq:driftdiff} is negative and proportional to the distance $l-k$, then Lemma \ref{lem:linear_drift} guarantees that the walkers will meet quickly. This explains the condition on the derivative of the speed in the next lemma.

\begin{lemma}\label{lem:meeting}
	Let $W_t^1$ and $W_t^2$ be two independent realizations of the random walk $W_t$. Let $W_0^1 = -a_n$, $W_0^2 = a_n$ and let $T = \inf\{t:W_t^1=W_t^2\}$. Assume that there exists a constant $d>0$ and a differentiable function $h:\mathbb{R}\to \mathbb{R}$ such that $\lambda_k-\mu_k = h(k)$ for all $k\in \{-a_n,\ldots,a_n\}$ and
	\begin{align}
	\frac{d}{dx}h(x) \leq -d \qquad\text{for}\qquad x\in [-a_n,a_n].
	\end{align}
	Then for all $t\geq 0$ 
	\begin{align}
	\mathbb{P}(T>t) \leq 2a_n\cdot e^{-d\cdot t}.
	\end{align}
\end{lemma}
\begin{proof}
We introduce the two-dimensional process $\{P_t,t\geq 0\}$ describing the pair of walkers and couple them as soon as they meet:
\begin{align}\label{eq:Pdef}
P_t = \left(P_t^1,P_t^2\right) = \begin{cases}
\big(W_t^1,W_t^2\big) & t<T,\\
\big(W_t^1,W_t^1\big) & t\geq T.
\end{cases}
\end{align}
When $t<T$, the process $P_t$ is in some state $(k,l)$ with $k<l$. The transition rates then are
\begin{align}\label{eq:Prates}
\begin{array}{lll}
(k,l) \to (k-1,l) \text{\ with rate\ } \mu_k, && (k,l) \to (k,l-1) \text{\ with rate\ } \mu_l, \\
(k,l) \to (k+1,l) \text{\ with rate\ } \lambda_k, & & (k,l) \to (k,l+1) \text{\ with rate\ } \lambda_l. 
\end{array}
\end{align}
The distance $l-k$ between the walkers increases at rate $\mu_k+\lambda_l$ and decreases at rate $\lambda_k+\mu_l$. We aim to apply Lemma \ref{lem:linear_drift} to the distance, and therefore look at the difference of these rates:
\begin{align}
(\lambda_l+\mu_k)-(\lambda_k+\mu_l) &= (\lambda_l-\mu_l)-(\lambda_k-\mu_k) \nonumber\\
& = f(l)-f(k) = \int_{k}^l \frac{d}{dx}f(x)dx \leq -d\cdot (l-k).
\end{align}
Unfortunately, the left hand side is not necessarily determined by $l-k$, so the distance $P_t^2-P_t^1$ is not a Markov process. We therefore add a third walker $\{P_t^3,t\geq 0\}$ with $P_0^3=a_n$ and satisfying $P_t^1\leq P_t^2\leq P_t^3$ for all $t\geq 0$. We will define its rates such that the distance $D_t:=P_t^3-P_t^1$ is a Markov process (but $P_t^3$ itself is not Markovian).

The rate to the right of the new walker $P_t^3$ will depend on the position of the walker $P_t^1$. If $P_t^1 = k$ and $P_t^3 = m$, then its rate to the right is chosen to be
\begin{align}
\nu_{k,m} := -d\cdot (m-k)+\lambda_k+\mu_m-\mu_k\geq \lambda_m. 
\end{align}
Suppose $(P_t^1,P_t^2,P_t^3) = (k,l,m)$ with $l<m$. The transitions in the first and second coordinate still follow the rules given in \eqref{eq:Pdef} and \eqref{eq:Prates}. The third coordinate increases and decreases at the following rates:
\begin{align}
\begin{array}{lll}
(k,l,m) \to (k,l,m+1)\phantom{ruimte} &&\text{\ with rate\ } \nu_{k,m},\\
(k,l,m) \to (k,l,m-1) &&\text{\ with rate\ } \mu_m.
\end{array}
\end{align}
If $P_t^2=P_t^3$, they can either move simultaneously, or only $P_t^3$ moves to the right. I.e., if $(P_t^1,P_t^2,P_t^3) = (k,l,l)$ with $k< l$, the possible transitions are
\begin{align}
\begin{array}{lll}
(k,l,l) \to (k,l-1,l-1)\phantom{ruimte} && \text{\ with rate\ } \mu_l\\
(k,l,l) \to (k,l+1,l+1) && \text{\ with rate\ } \lambda_l\\
(k,l,l) \to (k,l,l+1) &&\text{\ with rate\ } \nu_{k,l}-\lambda_l
\end{array}
\end{align}
Note that in any state $(P_t^1,P_t^2,P_t^3) = (k,l,m)$, the marginal rates of the first two walkers are as before. Define 
\begin{align}
\tilde T = \inf\{t:P_t^1 = P_t^3\}
\end{align}
and set $P_t = (W_t^1,W_t^1,W_t^1)$ for $t\geq \tilde T$. For $t<\tilde T$ and $(P_t^1,P_t^2,P_t^3) = (k,l,m)$, the distance $D_t=m-k$ decreases at rate $\lambda_k+\mu_m$ and increases at rate $\mu_k+\nu_{k,m}$. Hence, the speed of the difference is given by
\begin{align}
\mu_k+\nu_{k,m}-(\lambda_k+\mu_m) &= -d\cdot (m-k).
\end{align}
This means that $D_t$ behaves as a birth death process on $[2a_n]$, with negative linear speed and we can apply Lemma \ref{lem:linear_drift}. We obtain
\begin{align}
\mathbb{P}(\tilde{T}>t)\leq2a_n\cdot e^{-d\cdot t},
\end{align}
for all $t>0$. Since $\tilde{T}\geq T$, it follows that $\mathbb{P}(T>t)\leq \mathbb{P}(\tilde{T}>t)$ and this completes the proof.
\end{proof}

The previous result will be used to prove quick convergence of $W_t$ to the stationary distribution for any initial distribution $W_0$. The idea is the following: if two walkers meet, they will obey the same probability distribution after meeting. If a third walker starts in between and stays in between through coupling, they will all three meet and share the same distribution. Letting the  third walker start in the stationary distribution, the common distribution after meeting has to be the stationary distribution as well. 

\begin{proposition}\label{prop:mixing} Consider the walk $\{W_t,t\geq 0\}$ with arbitrary initial distribution $W_0$ and conditions on the rates as in Lemma \ref{lem:meeting}. Let $W_\infty$ have the stationary distribution. For all $t\geq 0$ and all $-a_n\leq k\leq a_n$, 
	\begin{align}
	\left\lvert \mathbb{P}(W_t\leq k) - \mathbb{P}(W_\infty\leq k)\right\rvert \leq 2a_ne^{-d\cdot t}.
	\end{align}
\end{proposition}

\begin{proof}
We consider four independent realizations of $\{W_t,t\geq 0\}$ with different initial distributions. 
\begin{itemize}
\item
Let $W_t^S$ be a realization of $W_t$ which starts in the stationary distribution, i.e. $\mathbb{P}(W^S_0 = k) = \mathbb{P}(W_\infty = k)$. By definition of stationarity, this means that 
\begin{align}
\mathbb{P}(W_t^S= k) = \mathbb{P}(W_\infty = k),
\end{align}
for all $t\geq 0$ and all $k$.
\item Let $\tilde W_t$ be a realization with arbitrary initial distribution $\tilde W_0$ on $\{-a_n,\ldots,a_n\}$.
\item 
 Let $W_t^1$ and $W_t^2$ be two more realizations with deterministic initial distributions $W_0^1\equiv -a_n$ and $W_0^2\equiv a_n$. 
\end{itemize}
At time $t=0$ we have $W_0^1\leq W_0^S\leq W_0^2$ and $W_0^1\leq \tilde W_0\leq W_0^2$. If any pair of walkers meets, we couple them. This implies that 
\begin{align}
W_t^1\leq W_t^S \leq W_t^2\qquad \text{and}\qquad W_t^1\leq \tilde W_t \leq W_t^2
\end{align}
for all $t\geq 0$. Lemma \ref{lem:meeting} gives us that $\mathbb{P}(T>t)\leq 2a_ne^{-d\cdot t}$, where $T = \inf\{t:W_t^1=W_t^2\}$ is the first meeting time of $W_t^1$ and $W_t^2$. It follows that 
\begin{align}
W_t^1=W_t^S= \tilde W_t=W_t^2
\end{align}
for all $t\geq T$. Observe that
\begin{align}
\mathbb{P}(\tilde W_t\leq k) &= \mathbb{P}(\tilde W_t\leq k\mid T\leq t)\cdot \mathbb{P}(T\leq t) + \mathbb{P}(\tilde W_t\leq k\mid T> t)\cdot \mathbb{P}(T> t)\nonumber\\
&\geq  \mathbb{P}(W_t^S\leq k\mid T\leq t)\cdot \mathbb{P}(T\leq t) \nonumber \\
&\geq (1-2a_ne^{-d\cdot t})\cdot \mathbb{P}(W_\infty\leq k)\nonumber\\ &\geq \mathbb{P}(W_\infty\leq k) - 2a_ne^{-d\cdot t}.
\end{align}
Similarly,
\begin{align}
\mathbb{P}(\tilde W_t\leq k) &= \mathbb{P}(\tilde W_t\leq k\mid T\leq t)\cdot \mathbb{P}(T\leq t) + \mathbb{P}(\tilde W_t\leq k\mid T> t)\cdot \mathbb{P}(T> t)\nonumber\\
&\leq  \mathbb{P}(W_t^S\leq k\mid T\leq t)\cdot \mathbb{P}(T\leq t) +  \mathbb{P}(T> t) \nonumber\\
&\leq  \mathbb{P}(W_\infty\leq k) + 2a_ne^{-d\cdot t},
\end{align}
which is sufficient to complete the proof.
\end{proof}

\section{Main result for walks on $\mathbb{Z}$}\label{sec:Zwalks}

So far, we investigated the random walks $W_t^{(n)}$ on a finite interval in $\mathbb{Z}$. The main conclusions are that these walks quickly reach their stationary distribution (Proposition \ref{prop:mixing}) and that these stationary distributions with a proper scaling converge to a Gaussian for $n\to\infty$ (Theorem \ref{thm:statdistr}). Observe that for $n$ fixed we obtain a fixed stationary distribution, i.e.\! we can send $t$ to $\infty$ without scaling the position by a function of time. In that sense, our convergence is of a different nature than for instance the central limit theorem for simple symmetric random walk.

We now aim to expand the state space of the random walk to $\mathbb{Z}$. This walk on $\mathbb{Z}$ will be denoted $M_t^{(n)}$. We take the rates inside $[-a_n,a_n]$ as before, but we allow the walk to continue beyond this interval as a simple random walk with arbitrary rates. If we start in $-a_n<k<a_n$, the walk on $\mathbb{Z}$ is indistinguishable from the walk restricted to $[-a_n,a_n]$, as long as we do not hit the boundaries. The typical phenomenon is that the walk still quickly starts to behave like a Gaussian, until (after a long time) the boundary is hit. To formalize and prove this, we first state the next result, which bounds the hitting time of the boundary in terms of the short term drift. For convenience we assume that $a_n$ is even, and write $M_t$ for $M_t^{(n)}$. 

\begin{proposition}\label{prop:hitboundary} 
	Consider the walk $\{M_t,t\geq0\}$. Suppose the initial distribution $M_0$ is supported on $[-\tfrac{a_n}{2},\tfrac{a_n}{2}]$. Let $T_n = \inf\{t\geq 0:|M_t|=a_n\}$ and let
	\begin{align}
	\rho_n=\max\left\{\left(\frac{\lambda_k}{\mu_k}\right)^{\text{sgn}(k)} \Bigm| \frac{a_n}{2}\leq |k|\leq a_n\right\},\quad \kappa_n = \max\left\{\lambda_k+\mu_k\Bigm| \frac{a_n}{2}\leq |k|\leq a_n\right\}.
	\end{align}
	Suppose $\rho_n<1$. Then, for $\varepsilon>0$ arbitrary, there exist $C_1, C_2 >0$ such that
	\begin{align}
	\mathbb{P}(T_n<t) \leq (2+\varepsilon)\kappa_nt\rho_n^{a_n/2-1},
	\end{align}
	whenever $\kappa_nt\geq C_1+C_2\cdot a_n\log(\rho_n^{-1})$.
\end{proposition}
  
\begin{proof}
Let $T_n^- = \inf\{t\geq 0:M_t=-a_n\}$ and $T_n^+ = \inf\{t\geq 0:M_t=a_n\}$. Then $T_n = \min\{T_n^-,T_n^+\}$. 

It is straightforward to define a birth-death process $Y^-_t$ and couple it to $M_t$ such that $Y^-_0=a_n/2$ and $Y^-_t\leq \min\{M_t+a_n,a_n/2\}$ for all $t$. So $Y_t^-$ is always less than the distance of $M_t$ to the left boundary of $[-a_n,a_n]$. Consequently, the extinction time of $Y_t^-$ is less than or equal to $T_n^-$. Similarly, we can construct a birth-death process $Y_t^+$ for which $Y_0^+=a_n/2$ and $Y_t^+\leq \min\{a_n-M_t,a_n/2\}$ for all $t$. The process $Y_t^+$ has extinction time less than or equal to $T_n^+$.

We bound the extinction times of $Y_t^-$ and $Y_t^+$ by Proposition \ref{prop:extinctioncontinuous}.  It then follows that for any $\varepsilon>0$, there exist $C_1,C_2>0$ such that
\begin{align}
\mathbb{P}\left(T_n^-<t\right)\leq (1+\varepsilon)\kappa_nt\rho^{a_n/2-1}\quad\text{and}\quad \mathbb{P}\left(T_n^+<t\right)\leq (1+\varepsilon)\kappa_nt\rho^{a_n/2-1},
\end{align}
whenever $\kappa_n t\geq C_1+C_2 a_n\log(\rho_n^{-1})$.
This implies the result of the proposition. 
\end{proof}

We collect and summarize the results and conditions on the random walks $\{M_t^{(n)},t\geq 0\}$. For each $n$, $M_t^{(n)}$ is a simple random walk with state space $\mathbb{Z}$. The rates in state $k$ are denoted $\mu_k^{(n)}$ (to the left) and $\lambda_k^{(n)}$ (to the right). 

Take $\sigma_n\to\infty$ and let $a_n \in (\sigma_n,\sigma_n^2)$ be an integer sequence satisfying \eqref{eq:defbn}. We assume existence of $n_0\in\mathbb{N}$ such that the rates satisfy the following conditions for $n\geq n_0$. These conditions only restrict the rates for $|k|\leq a_n$, and are sufficient for our main theorem.
\begin{enumerate}
	\item (Condition for quasi-stationary behaviour to be Gaussian.) Rates satisfy
\begin{align}\label{eq:cond1}
\log\left(\frac{\lambda_k^{(n)}}{\mu_{k+1}^{(n)}}\right) = -\frac{k}{\sigma_n^2}+\delta_n(k),\qquad -a_n\leq k \leq a_n-1,
\end{align}
(or the version in \eqref{eq:rates2}) with $\sum_k|\delta_n(k)|\leq Ka_n/\sigma_n^2$ for some constant $K$.

\paragraph{Remark.} This condition says that there is a long term drift towards 0 which is approximately linear in $k$. A condition of this type is needed for quasi-stationary behaviour to be Gaussian, since the long term drift determines the quasi-stationary distribution.
\item (Condition for quick convergence to quasi-stationary behaviour.) For each $n$, there exists a constant $d_n>0$ and a differentiable function $h_n:\mathbb{R}\to\mathbb{R}$ such that $\lambda_k^{(n)}-\mu_k^{(n)}=h_n(k)$ for all $-a_n\leq k\leq a_n$ and
\begin{align}\label{eq:cond2}
\frac{d}{dx}h_n(x)\leq -d_n,\qquad x\in [-a_n,a_n].
\end{align}

\paragraph{Remark.}  This is a condition on the speed of the process. From condition \eqref{eq:cond1}, it does not yet follow that there is a speed towards 0. See also the discussion Section \ref{sec:drift}. A speed towards 0 is necessary to control the time to reach 0. This speed is allowed to go to $0$ as $n\to\infty$. However, we only get a meaningful result if the speed does not vanish too quickly, see the theorem below.

\item (Condition for long persistence of quasi-stationary behaviour.) For some $\varepsilon>0$, independent of $n$, the constant
\begin{align}\label{eq:constants1}
\rho_n = \max_{\frac{a_n}{2}\leq |k|\leq a_n} \left(\frac{\lambda_k^{(n)}}{\mu_k^{(n)}}\right)^{\text{sgn}(k)}
\end{align}
satisfies $e^{-1}<\rho_n<1-\frac{(2+\varepsilon)\log(a_n)}{a_n}$. 

\paragraph{Remark.}  Condition \eqref{eq:cond2} on the speed already implies that there is a short term drift towards 0, i.e. that $\rho_n<1$. If the rates are not too wild, it is natural that $\rho_n$ goes to $1$ as $n\to\infty$. The lower bound on $\rho_n$ is therefore not really restrictive, but is there for technical reasons. The upper bound ensures that the short term drift is strong enough. Essentially the short term drift says how likely it is to move away from 0, and therefore determines the time spent in the quasi-stationary distribution. The condition guarantees that $\rho_n^{-a_n/2}$ appearing in the theorem below is at least $a_n^{(2+\varepsilon)/2}$ so that the length of $I_n$ will at least grow like a positive power of $a_n$. 
\end{enumerate}
We further introduce the constant 
\begin{align}\label{eq:constants2}
\kappa_n = \max_{\frac{a_n}{2}\leq |k|\leq a_n} (\lambda_k^{(n)}+\mu_k^{(n)}).
\end{align}
Observe that 
\begin{align}\label{eq:observation_kappa}
2\kappa_n &\geq |h_n(a_n)|+|h_n(-a_{n})|\geq h_n(-a_n)-h_n(a_n)\geq 2d_na_n.
\end{align}
Under the three conditions above, our main result is the following. 


\begin{theorem}\label{thm:main}
Let $Z\sim N(0,1)$. Let the walks $M_t^{(n)}$ satisfy the three conditions above. Suppose $M_0^{(n)}$ satisfies $\mathbb{P}(-a_n/2\leq M^{(n)}_0\leq a_n/2) =1$ for all $n$. Let $c_n$ be any sequence with $\lim_{n\to\infty}c_n=0$. Let
\begin{align}
I_n = \left\{t\geq 0:d_n^{-1}c_n^{-1}\log(a_n)\le t\le c_n\kappa_n^{-1}\rho_n^{-a_n/2}\right\}.
\end{align}
If $I_n\neq\emptyset$, then 
\begin{align}\label{eq:thmconclusion}
\lim_{n\to\infty}\sup_{t\in I_n} \left|\mathbb{P}\left(\tfrac{M_t^{(n)}}{\sigma_n}\leq x\right)-\mathbb{P}(Z\leq x)\right|=0, \qquad\text{for\ all}\ x\in\mathbb{R}.
\end{align}	
\end{theorem}

\begin{remark} 
To maximize the length of $I_n$, take for $c_n$ a slowly decreasing function. Still, $I_n$ will be empty if $d_n$ goes to $0$ too fast. How fast $d_n$ can go to $0$ depends on the other parameters, but in any case $d_n=\Omega(\kappa_n/a_n)$ suffices to get $|I_n|\geq a_n^\delta$ for some $\delta>0$.
\end{remark}

\begin{proof}[Proof of Theorem \ref{thm:main}.]
Let $W_t^{(n)}$ be the random walk obtained by restricting $M_t^{(n)}$ to $[-a_n,a_n]$. This means that the two walks have the same rates, except in the boundary points $-a_n$ and $a_n$, where we set the outward rates of $W_t^{(n)}$ equal to 0. Let $W_\infty^{(n)}$ be the stationary distribution of $W_t^{(n)}$. Then, for $k\in\mathbb{Z}$, 
\begin{align}
\Big|\mathbb{P}(&\left.\!\!M_t^{(n)}\leq k)-\mathbb{P}(W_\infty^{(n)}\leq k)\right|\nonumber \\
 &\leq \left|\mathbb{P}(M_t^{(n)}\leq k)-\mathbb{P}(W_t^{(n)}\leq k)\right|+\left|\mathbb{P}(W_t^{(n)}\leq k)-\mathbb{P}(W_\infty^{(n)}\leq k)\right|. \label{eq:probdiff}
\end{align}
For $t$ not too large, it is unlikely that $M_t^{(n)}$ already has reached the boundary of $[-a_n,a_n]$. Therefore, the first term on the right hand side of \eqref{eq:probdiff} will be small. For $t$ not too small, the walk $W_t^{(n)}$ will be close to its stationary distribution so that the second term will be small. We will make this precise, and find an interval for $t$ in which both terms are small.

Let $T_n = \inf\{t\geq 0:|M_t^{(n)}| = a_n\}$. Then we can couple the walks in such a way that $M_t^{(n)}=W_t^{(n)}$ for $t\leq T_n$. The goal now is to use Proposition \ref{prop:hitboundary} to show that $T_n$ is likely to be large. The conditions on $\rho_n$ give that for some $\delta>0$,
\begin{align}\label{eq:rhocond}
-1<\log(\rho_n)<-\frac{(2+\delta)\log(a_n)}{a_n}.
\end{align}
Since $c_n^{-1}\log(a_n)\to \infty$, the left hand side implies that $c_n^{-1}\log(a_n)\gg\log(\rho_n^{-1})$. Using \eqref{eq:observation_kappa}, we obtain for $t\in I_n$,
\begin{align}\label{eq:kappat1}
\kappa_n t \geq c_n^{-1}a_n\log(a_n) \gg a_n\log(\rho_n^{-1}).
\end{align}
Combining this with the right hand side of \eqref{eq:rhocond} gives
\begin{align}\label{eq:kappat2}
\kappa_n t > a_n\log(\rho_n^{-1}) > (2+\delta)\log(a_n) > \log(\sigma_n)\to\infty.
\end{align}
Take $\varepsilon>0$ arbitrary and choose $C_1$ and $C_2$ as in Proposition \ref{prop:hitboundary}. For $n$ large enough, by \eqref{eq:kappat1} and \eqref{eq:kappat2} we get that $\kappa_n t \geq C_1+C_2\cdot a_n\log(\rho_n^{-1})$, and hence we find
\begin{align}
\left|\mathbb{P}(M_t^{(n)}\leq k)-\mathbb{P}(W_t^{(n)}\leq k)\right| \leq& \left|\mathbb{P}(M_t^{(n)}\leq k,T_n\geq t)-\mathbb{P}(W_t^{(n)}\leq k, T_n\geq t)\right|\nonumber\\
& +\left|\mathbb{P}(M_t^{(n)}\leq k,T_n< t)-\mathbb{P}(W_t^{(n)}\leq k, T_n< t)\right|\nonumber\\
\leq &\ \mathbb{P}(T_n<t)\leq (2+\varepsilon)\kappa_nt\rho_n^{a_n/2-1},
\end{align}
For all $t\in I_n$, we get
\begin{align}
\left|\mathbb{P}(M_t^{(n)}\leq k)-\mathbb{P}(W_t^{(n)}\leq k)\right| \leq   (2+\varepsilon)\kappa_nt\rho_n^{a_n/2-1} \leq (2+\varepsilon)ec_n \stackrel{n\to\infty}{\longrightarrow} 0.
\end{align}
For the second term in \eqref{eq:probdiff}, Proposition \ref{prop:mixing} gives
\begin{align}
\left|\mathbb{P}(W_t^{(n)}\leq k)-\mathbb{P}(W_\infty^{(n)}\leq k)\right| \leq 2a_ne^{-d_nt}. 
\end{align} 
Taking $t\in I_n$ again, we obtain
\begin{align}
\left|\mathbb{P}(W_t^{(n)}\leq k)-\mathbb{P}(W_\infty^{(n)}\leq k)\right| \leq 2a_ne^{-c_n^{-1}\log(a_n)} = 2a_n^{1-c_n^{-1}} \stackrel{n\to\infty}{\longrightarrow} 0.
\end{align}
The upper bounds for the terms in \eqref{eq:probdiff} are independent of $t$, so that
\begin{align}
\lim_{n\to\infty}\sup_{t\in I_n}\Big|\mathbb{P}(&\left.\!\!M_t^{(n)}\leq k)-\mathbb{P}(W_\infty^{(n)}\leq k)\right|=0.
\end{align}
Theorem \ref{thm:statdistr} yields for all $x\in\mathbb{R}$
\begin{align}
\left|\mathbb{P}(Z\leq x)-\mathbb{P}(W_\infty^{(n)}/\sigma_n\leq x)\right|\to 0,
\end{align}
from which we conclude that
\begin{align}
\lim_{n\to\infty}\sup_{t\in I_n} \left|\mathbb{P}(M_t^{(n)}/\sigma_n\leq x)-\mathbb{P}(Z\leq x)\right|=0.
\end{align}
\end{proof}

\subsection{An example}

In this section we illustrate Theorem \ref{thm:main} by simulation of a simple example. Let $\sigma_n = \sqrt{n}$ and $a_n = \lceil \sigma_n^{3/2}\rceil=\lceil n^{3/4}\rceil$. Consider a sequence of random walks $\left\{M_t^{(n)},t\geq 0\right\}$ on $\mathbb{Z}$ with rates
\begin{align}
\lambda_k^{(n)} = 1-\frac kn,\qquad \mu_{k}^{(n)} = 1, 
\end{align}
for $-a_n\leq k\leq a_n$. For the sake of illustration, we take absorbing boundaries, i.e. if the walk leaves $[-a_n,a_n]$ and hits $-a_n-1$ or $a_n+1$, it stops moving. 
Then condition 1 for the theorem is satisfied. Defining $h_n:\mathbb{R}\to\mathbb{R}$ by $h_n(x) = -\frac xn$, the difference of the rates is given by 
\begin{align}
\lambda_k^{(n)} - \mu_{k}^{(n)} = -\frac kn = h_n(k). 
\end{align}
Letting $d_n = \frac 1n$, the derivative of $f_n$ is 
\begin{align}
\frac{d}{dx} f_n(x) = -\frac 1n = -d_n,
\end{align}
so that condition 2 is satisfied as well. Taking $\rho_n$ as in \eqref{eq:constants1}, we obtain
\begin{align}
\rho_n-1 &= -1+\max_{\tfrac{a_n}{2}\leq |k| \leq a_n} \left(1-\frac kn\right)^{\text{sgn}(k)} 
\sim -\tfrac12 a_n^{-1/3}\sim -\tfrac12n^{-1/4},
\end{align}
satisfying condition 3. For $\kappa_n$ we find
\begin{align}
\kappa_n = \max_{\tfrac{a_n}{2}\leq |k| \leq a_n} (\lambda_k^{(n)}+\mu_k^{(n)}) =  \max_{\tfrac{a_n}{2}\leq |k| \leq a_n} \left(2-\frac kn\right) = 2+\frac{a_n}{n} \sim 2+n^{-1/4}.
\end{align}
In this example, the conclusion \eqref{eq:thmconclusion} of the theorem holds for any $c_n\to 0$ and
\begin{align}
I_n = \{t\geq 0: c_n^{-1}n\log(n)\leq t\leq c_ne^{\sqrt{n}/4}\}.
\end{align}
In this time interval, for $n$ large, the distribution of the walk $M_t^{(n)}$ is very close to a Gaussian. In fact, simulations show that convergence to a Gaussian distribution is probably still true in much larger intervals. The constant $1/4$ in the exponent of the upper bound of $I_n$ can presumably be improved by a more subtle estimation of the short term drift (compare with \eqref{eq:constants1} and the corresponding remark).

\begin{figure}
\includegraphics[width=\linewidth]{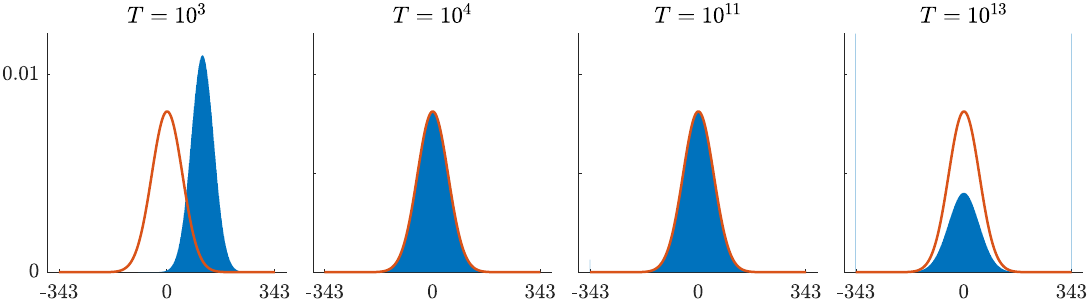}\caption{Distribution of $M_T^{(7^4)}$ for four different times $T$ compared to the probability density function of $N(0,7^4)$.}\label{fig:walk}
\end{figure}

We take a look at this process for $n=7^4 = 2401$. Then $a_n = 7^3 = 343$ and $\sigma_n = 7^2 = 49$. The walk moves between $\pm 344$ and stops when it hits one of these boundaries. We take initial condition $M_0^{(n)} = 344/2 = 172$. In view of Theorem \ref{thm:main}, we expect the distribution of $M_t^{(n)}$ to quickly approach a $N(0,\sigma_n^2)$ distribution and stay close to it for a long time. 

This is indeed exactly what we observe, see Figure \ref{fig:walk}. We numerically calculated the probability distribution of $M_T^{n}$ for four different values of $T$. For $T=10^3$, the influence of the initial condition $M_0^{n} = 172$ is still clearly visible: the expection of $M_T^{(n)}$ is about 113. For $T=10^4$, the expectation has dropped to 2.7. Letting $f_n$ be the probability density of the $N(0,\sigma_n^2)$, the distance 
\begin{align}
\sum_{k=-344}^{344}|\mathbb{P}(M_T^{(n)} = k) - f_n(k)|
\end{align} 
is already as small as 0.05. Then for a long time, the distribution is only very slightly changing. For $T=10^{11}$, the expected position is still very close to 0 (approximately -0.2), while the distance to the Gaussian in the above measure is 0.006. Finally, for $T=10^{13}$, the probability that the walk is absorbed at one of the boundaries is 0.505 and the distance to the Gaussian is 1.01 (essentially twice the absorption probability). Conditioned on the event that it is not yet absorbed, we still recognize the Gaussian shape. Somewhat surprising at first sight, the expected position is now -140. This indicates that absorption on the left is much more likely than absorption on the right.

\section{Application: the contact process on $K_n$}\label{sec:contactprocess}

In this section we aim to apply our results to the contact process on the complete graph $K_n$. In this graph each node is either infected or healthy. If it is infected, it is healing at rate 1. Furthermore, it is infecting each of its healthy neighbors at rate $\tau$. The number of infected individuals is a simple birth-death process $\{\tilde X_t\}_{t\geq 0}$ on $[n]$ with transition rates 
\begin{eqnarray}
k\rightarrow k+1 & \text{with rate} & k(n-k)\tau \\
k\rightarrow k-1 & \text{with rate} & k 
\end{eqnarray}
for all $k=0,\ldots,n$. Take $\tau=\frac\lambda n$ with constant $\lambda>1$ and let $k_0 = \lceil(1-\frac 1 \lambda)n\rceil = (1-\frac{1}{\lambda})n+\varepsilon$ with $\varepsilon\in [0,1)$. Note that $k_0$ corresponds to the point where the rates up and down are (approximately) balanced. Now consider the process $X_t := \tilde X_t - k_0$. This is a random walk on $[-k_0,n-k_0]$ which is attracted to $0$ and to which we will apply Theorem \ref{thm:main}. Let the rates up and down of this process be called $\lambda_k$ and $\mu_k$ respectively. Then for $-k_0 \leq k\le n-k_0$,
\begin{align}
\lambda_k = (k+k_0)(n-(k+k_0))\frac\lambda n,\qquad\mu_k = k+k_0.\label{eq:contactrates}
\end{align}
After some calculations, we obtain for $-k_0\leq k\leq n-k_0-1$ the ratio
\begin{align}
\frac{\lambda_k}{\mu_{k+1}} &= 1-\frac{k\lambda} n+\frac{(k+k_0-n)\lambda}{(k+k_0+1)n} -\varepsilon\frac\lambda n. 
\end{align}
Since $k_0$ is of the same order as $n$, the last two terms will be small if $k$ is small compared to $n$. In the first two terms, we recognize that the ratio $\lambda_k/\mu_{k+1}$ has the desired form (as in \eqref{eq:rates2}) for Gaussian quasi-stationary behavior. Indeed, we let 
\begin{align}\label{eq:sigmaa}
\sigma_n = \sqrt{\frac{n}{\lambda}}
\end{align}
and suppose that $a_n$ satisfies 
\begin{align}\label{eq:a}
a_n = o(\sigma_n^2) = o(n),\qquad a_n\geq \sigma_n\sqrt{(8+\delta)\log(\sigma_n)} 
\end{align}
for some $\delta>0$, so that \eqref{eq:defbn} holds. Take $-a_n \leq k \leq a_n-1$. Then we can write
\begin{align}
\frac{\lambda_k}{\mu_{k+1}} 
&= 1-\frac{k\lambda}{n}+\frac{\left(-(\lambda-1)\varepsilon-1\right)n+O(a_n)}{(1-\tfrac 1\lambda)n^2+O(n a_n)}\\
&= \left(1-\frac{k\lambda}{n}\right)\left(1+\delta_n\right)
\end{align}
with $|\delta_n| \leq \frac{\lambda^2}{(\lambda-1)n}+o(n^{-1}) = \frac{\lambda/(\lambda-1)}{\sigma_n^2}+o(\sigma_n^{-2})$. This shows that the first condition for Theorem \ref{thm:main} is satisfied for $n$ large enough. Therefore, the contact process on $K_n$ can already be predicted to have a quasi-stationary distribution close to a Gaussian with variance $\sigma_n^2 = n/\lambda$ and expectation $(1-\frac 1\lambda)n$. 

We proceed to check the second condition. 
Define for $x\in\mathbb{R}$
\begin{align}\label{eq:h} 
h_n(x) = (\lambda - 1)(x+k_0)-\frac\lambda n (x+k_0)^2.
\end{align} 
Then $h_n(k) = \lambda_k-\mu_k$. By noting that $h_n(x) = -(\lambda-1)(x+\varepsilon) -\frac{\lambda}{n}(x+\varepsilon)^2$, it follows that the zeros of $h_n$ are at $x = -k_0$ and $x = -\varepsilon$.
 The former corresponds to extinction, the latter to the expectation of the quasistationary distribution. For $x\in [-a_n,a_n]$ we have
\begin{align}
\frac{d}{dx}h_n(x) &= 1-\lambda-\frac{2\lambda}{n}(x+\varepsilon)\leq 1-\lambda+\frac{2\lambda a_n}{n}\nonumber\\
&= 1-\lambda + o(1),\label{eq:d} 
\end{align}
which is negative for $n$ large enough, since $\lambda>1$. This guarantees that $X_t$ will quickly reach it quasi-stationary distribution if it is in the interval $[-a_n/2,a_n/2]$ at time 0. This is a bit restrictive, but we will relax this condition later on. 

We now look at the third condition, which makes it hard to escape from the quasi-stationary behavior. The short term drift of the process is
\begin{align}\label{eq:contactratio}
\frac{\lambda_k}{\mu_k} = 1-\frac{k\lambda}{n}-\varepsilon\frac{\lambda}{n}.
\end{align}
Therefore, using the upper and lower bounds for $a_n$ in \eqref{eq:a}, we obtain
\begin{align}
\rho_n -1 &= -1+\max_{\frac{a_n}{2}\leq |k| \leq a_n} \left(\frac{\lambda_k}{\mu_k}\right)^{\text{sgn}(k)} \sim -\frac{a_n}{2}\frac{\lambda}{n} = -\frac{a_n}{2\sigma_n^2}
\leq -(2+\tilde\delta)\frac{\log(a_n)}{a_n},\label{eq:rho}
\end{align}
for some $\tilde\delta>0$, which means that the third condition is satisfied as well and consequently all conditions on the rates are fulfilled. 

Theorem \ref{thm:main} now can be applied to the process $X_t$, provided that $X_0$ is close enough to 0. For the contact process, this means we have to start close to the equilibrium point $k_0$. The next lemma will be used to relax this initial condition. If there are enough infected individuals initially, then the number of infected is likely to quickly be close to $k_0$, after which we can apply Theorem \ref{thm:main}.

\begin{lemma}\label{lem:contactprocess}
	Consider the contact process $\{\tilde X_t\}_{t\geq 0}$ on $K_n$ with healing rate $1$ and infection rate $\tau = \tfrac \lambda n$ for some $\lambda>1$. Let $c>\frac{2\lambda+2}{\lambda-1}$ be a constant and suppose 
	\begin{align}
	\tilde X_0 \geq c\log(n).
	\end{align} 
	Let $k_0 = \left\lceil\frac{\lambda-1}{\lambda}n\right\rceil$ and $a_n$ as in \eqref{eq:a} and define
	\begin{align}
	T = \inf\{t\geq 0: k_0-\frac{a_n}{2}\leq \tilde X_t\leq k_0+\frac{a_n}{2}\}.
	\end{align}
	Then 
	\begin{align}
	\lim_{n\to\infty}\mathbb{P}\left(T > \frac{n}{2}\right) = 0.
	\end{align}
\end{lemma}

\begin{proof}
We distinguish two cases, $\tilde X_0\geq k_0$ and $\tilde X_0<k_0$. We will show that in the first case $\mathbb{P}(T>t)$ goes to 0 exponentially in $t$. The second case requires more carefulness, since $\tilde X_t$ could go extinct before reaching $k_0-\frac{a_n}{2}$. Hence, for $n$ fixed $\mathbb{P}(T=\infty)$ is strictly positive. Nevertheless it is still true that $\mathbb{P}(T>t)$ as a function of $t$ quickly decreases to $\mathbb{P}(T=\infty)$. The latter probability goes to 0 as $n$ increases.

\textbf{Case 1:} $\tilde X_0\geq k_0$.
 This case is easy, since the translated process $X_t = \tilde X_t-k_0$ can be dominated by a birth-death process with negative linear speed, see Section \ref{sec:neglindrift}. In this case we can write
\begin{align}
T = \inf\{t\geq 0: \tilde X_t - k_0 \leq \frac{a_n}{2}\} = \inf\{t\geq 0: X_t \leq  \frac{a_n}{2}\}. 
\end{align}
 Let $T^+ = \inf\{t\geq 0: X_t =0 \}$ and define the absorbed process by
\begin{align}
X_t^+ = \begin{cases} X_t & t<T^+, \\ 0 & t\geq T^+.\end{cases}
\end{align}
Then $X_t^+$ is a proper birth-death process on $[n-k_0]$ with extinction time $T^+>T$. Taking $h_n(x)$ as in \eqref{eq:h}, we obtain for all $x>0$
\begin{align}
\frac{d}{dx} h_n(x) 
& = -(\lambda-1)-\frac{2\lambda}{n}(x+\varepsilon)<1-\lambda<0.
\end{align}
Since $h_n(0)\leq 0$, the speed of $X_t^+$ is at most a negative linear function of $k$:
\begin{align}
\lambda_k-\mu_k = h_n(k) \leq (1-\lambda)k,\qquad k\geq 0.
\end{align}
By Lemma \ref{lem:linear_drift} and a suitable coupling it follows that 
\begin{align}
\mathbb{P}(T>t) \leq \mathbb{P}(T^+>t) \leq (n-k_0)\cdot e^{(1-\lambda)t},
\end{align}
which goes to 0 for $t = \tfrac n 2$.

\textbf{Case 2:} $c\log(n)\leq \tilde X_0=m<k_0$. Our argument in this case roughly goes as follows. We will choose some $a\in [1,m]$ and we will show that the hitting time of $a$ is large by coupling with a birth-death process $Y_t$. For $t$ less than the hitting time of $a$, we can couple $|\tilde X_t-k_0|$ to a birth death process $Z_t$ with negative speed, which quickly dies out. The conclusion then is that with high probability $\tilde X_t$ hits $k_0-\frac{a_n}{2}$ before $a$ and that $T$ is small. See Figure \ref{fig:lemma_cp} for an illustration of the typical situation.

\begin{figure}
\includegraphics[width=\linewidth]{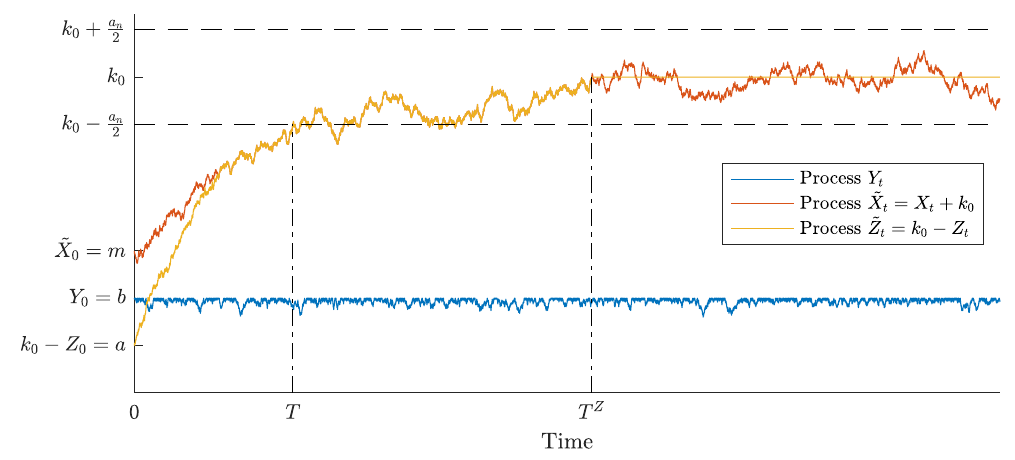}\caption{$Y_t$ is a birth death process on $[b]$. It hits $a$ at time $T_{b,a}$ (outside the plot window) and by coupling $Y_t\leq \tilde X_t$ for all $t$. In particular $Y_t$ hits $a$ before $\tilde X_t$ does, so $T_{b,a}\leq T_a$. $Z_t$ is a birth death process on $[k_0-a]$. The process $\tilde Z_t = k_0-Z_t$ hits $k_0-1$ at time $T^Z$. By coupling $\tilde Z_t\leq \tilde X_t$ for $t\leq \min\{T^Z,T_a\}$. Both $Y_t$ and $\tilde Z_t$ have a positive drift, so typically $T_a$ is much larger than $T^Z$.}\label{fig:lemma_cp}
\end{figure}

Let $a,b\in\mathbb{N}$  with $a\leq b\leq \min\{m,(k_0-\varepsilon)/2\}$ (remember that $k_0-\varepsilon = (1-1/
\lambda)n$). We will first use Proposition \ref{prop:extinctioncontinuous} to control the hitting time of $a$, defined by 
\begin{align}
T_a := \inf\{t\geq 0:\tilde X_t = a\} = \inf\{t\geq 0: X_t = a-k_0\}.
\end{align}
Let $\{Y_t,t\in\mathbb{R}\}$ be a birth-death process on $[b]$ with birth rate $\tilde \lambda_k = k(n-k)\lambda/n$ and death rate $\tilde \mu_k = k$. For $k=b$, we set the birth rate equal to 0. Let $Y_0 = b$ and $T_{b,a} = \inf\{t\geq 0:Y_t = a\}$. Clearly we can couple $Y_t$ and $\tilde X_t$ such that $Y_t\leq \tilde X_t$ for all $t$. Let 
\begin{align}
\tilde \kappa &= \max_{0\leq k\leq b}\tilde \lambda_k+\tilde \mu_k = \max_{0\leq k\leq b} (1+\lambda)k-\tfrac{\lambda}{n}k^2,\\
\tilde \rho &= \max_{0\leq k\leq b }\frac{\tilde \mu_k}{\tilde \lambda_k} =\max_{0\leq k\leq b } \frac{1}{\lambda(1-\frac k n)}. 
\end{align}
Then 
\begin{align}
b\leq \tilde \kappa \leq (1+\lambda) b,\qquad \frac{1}{\lambda}\leq \tilde \rho \leq \frac{1}{\lambda (1-\frac{k_0-\varepsilon}{2n})} = \frac{2}{\lambda+1}<1.
\end{align}
Proposition \ref{prop:extinctioncontinuous} applies to $Y_t$. Let $\delta>0$ and choose $C_1$ and $C_2$ as in that proposition.
Suppose $b\to\infty$ as $n\to\infty$ and let $C_3>C_2\log(\lambda)$. Then
\begin{align}
\tilde\kappa t \geq bt \geq C_1+C_2b\log(\lambda)\geq C_1+C_2b\log(\tilde\rho^{-1})
\end{align}
if $t>C_3$ and $n$ large enough, Proposition \ref{prop:extinctioncontinuous} then gives a bound on the extinction time of $Y_t$. In fact, we can bound $T_{a,b}$ analogously if $b-a\to\infty$ as $n\to\infty$. We obtain (for $t>C_3$ and $n$ large enough) that 
\begin{align}
\mathbb{P}(T_{b,a}<t) \leq C_4(b-a)t\left(\frac{2}{\lambda+1}\right)^{b-a},
\end{align}
with $C_4 = (1+\delta)(1+\lambda)$. By the coupling $Y_t \leq \tilde X_t$, we have $T_{b,a}\leq T_a$ so that 
\begin{align}\label{eq:Tabound}
\mathbb{P}(T_{a}<t) \leq C_4(b-a)t\left(\frac{2}{\lambda+1}\right)^{b-a}.
\end{align}

The second step in our argument involves another coupling. Consider the distance $|X_t|=k_0-\tilde X_t$ before $\tilde X_t$ hits $a$ or $k_0$. If $|X_t|=k$ for some $0< k\leq k_0-a$, then $X_t = -k$. Therefore, the distance increases at rate $\mu_{-k}$ and decreases at rate $\lambda_{-k}$. We will couple the distance to a birth-death process $\{Z_t,t\in\mathbb{R}\}$ on $[k_0-a]$, having birth rates $\lambda_k^Z$ and death rates $\mu_k^Z $ defined by
\begin{align}
\lambda_k^Z &:= \mu_{-k} = k_0-k,&&\lambda^Z_0 = \lambda^Z_{k_0-a}=0,\\
\mu_k^Z &:= \lambda_{-k} = (k_0-k)(n-(k_0-k))\frac{\lambda}{n}, &&\mu^Z_0=0.
\end{align}
Take initial condition $Z_0=k_0-a$, and let 
\begin{align}
T^Z = \inf\{t:Z_t = 1\}.
\end{align}
Couple such that $k_0-\tilde X_t\leq Z_t$ for all $t\leq\min\{T^Z,T_a\}$. If $T^Z\leq T_a$, then $T<T^Z$. Therefore, $T^Z$ satisfies $T^Z > \min\{T,T_a\}$. To bound $T^Z$, we will apply Lemma \ref{lem:mixing_differential}. First we calculate the speed
\begin{align}
\lambda_k^Z-\mu_k^Z &= (k_0-k)(\frac{n}{\lambda}-n+k_0-k)\frac{\lambda}{n}\\
& = (k_0-k)(\varepsilon-k)\frac{\lambda}{n} \leq -(k-1)(k_0-k)\frac{\lambda}{n}.\label{eq:lowerdrift}
\end{align} 
For $1\leq k\leq k_0-a$, this speed is negative and convex in $k$. It therefore can be upper bounded by a linear function which intersects the right hand side of \eqref{eq:lowerdrift} for $k=1$ and $k=k_0-a$. This gives
\begin{align}
\lambda_k^Z-\mu_k^Z \leq -\frac{a(k-1)\lambda}{n}.
\end{align}
We now find the differential inequality 
\begin{align}
\frac{d}{dt}\mathbb{E}[Z_t] \leq -\frac{a\lambda}{n}\cdot (\mathbb{E}[Z_t]-1).
\end{align}
It follows that 
\begin{align}
\mathbb{E}[Z_t]-1 \leq (k_0-a-1)\cdot e^{-\tfrac{a\lambda}{n}t},
\end{align}
and hence we obtain
\begin{align}\label{eq:TZbound}
\mathbb{P}(T^Z>t) = \mathbb{P}(Z_t-1\geq 1) \leq (k_0-a-1)\cdot e^{-\tfrac{a\lambda}{n}t}.
\end{align}

Now we can finish our argument by using \eqref{eq:Tabound} and \eqref{eq:TZbound} and the couplings between $\tilde X_t$, $Y_t$ and $Z_t$. For $n$ large enough,
\begin{align}
\mathbb{P}(T>t) &= \mathbb{P}(T>t,T_a<t)+\mathbb{P}(T>t,T_a\geq t)\\
&\leq \mathbb{P}(T_a<t)+\mathbb{P}(\min\{T,T_a\}\geq t)\\
&\leq \mathbb{P}(T_a<t)+\mathbb{P}(T^Z>t)\\
&\leq C_4(b-a)t\left(\frac{2}{\lambda+1}\right)^{b-a} + k_0\cdot e^{-\tfrac{a\lambda}{n}t}\label{eq:sombound}.
\end{align}
Let $a = \tfrac{c}{2}\log(n)$. Choose $b=2a$ and $t=\tfrac{n}{2}$. Then $b\leq m$ as required and $\frac{c}{2}>\tfrac{\lambda+1}{\lambda-1}$. For the first term in \eqref{eq:sombound} this gives
\begin{align}
\mathbb{P}(T_a<\tfrac{n}{2}) &\leq \frac{ C_4c}{4}\log(n)n^{1+\tfrac{c}{2}\log\left(\frac{2}{\lambda+1}\right)} \stackrel{n\to\infty}{\longrightarrow}0, 
\end{align}
using that $1+\tfrac{c}{2}\log(\tfrac{2}{\lambda+1})<0$ for $\lambda>1$.
For the second term in \eqref{eq:sombound}, since $c\lambda/4>1$, we obtain
\begin{align}
\mathbb{P}(T^Z>\tfrac n 2) &\leq k_0 e^{-c\lambda\log(n)/4} \leq n^{1-\tfrac{c\lambda}{4}}\stackrel{n\to\infty}{\longrightarrow} 0,
\end{align}
giving the desired conclusion.
\end{proof}

We are ready to state our metastability result for the contact process on the complete graph $K_n$ with infection rate $\tau=\lambda/n$ for some $\lambda>1$. At $t=0$, we can start with a quite small set of infected nodes (thanks to Lemma \ref{lem:contactprocess}). The number of infected individuals will rapidly go towards its equilibrium $\frac{\lambda-1}{\lambda}n$. In a large time window (from $t$ linear in $n$ to almost exponential in $n$), the number of infected individuals will behave like a Gaussian centered at $\frac{\lambda-1}{\lambda}n$ and with variance $n/\lambda$. The time window in the theorem is maximized by taking a function $g(n)$ which goes to 0 slowly. 

\begin{theorem}\label{thm:contactprocess}
Let $\lambda>1$. Consider for each $n\in\mathbb{N}$ the contact process $\{\tilde X_t^{(n)}\}_{t\geq 0}$ on the complete graph $K_n$ with healing rate $1$ and infection rate $\tau = \tfrac \lambda n$. Let $c>\frac{2(\lambda+2)}{\lambda-1}$ be a constant and suppose that the initial distributions satisfy
\begin{align}
\lim_{n\to\infty}\mathbb{P}(\tilde X_0^{(n)} \geq c\log(n)) = 1.
\end{align} 
Let $Z\sim N(0,1)$ and define $\mu = \frac{\lambda-1}{\lambda}n$ and $\sigma_n=\sqrt{\tfrac{n}{\lambda}}$. Let $g(n)$ be any function with $\lim_{n\to\infty}g(n) = 0$ and $g(n)\geq \log(n)/n$. Then for all $x\in\mathbb{R}$,
\begin{align}
\lim_{n\to\infty}\sup_{t\in \left[n,\exp(g(n)\cdot n)\right]}\left\lvert\mathbb{P}(\tfrac{\tilde X_t^{(n)}-\mu}{\sigma_n}\leq x)-\mathbb{P}(Z\leq x)\right\rvert =0.
\end{align}
\end{theorem}

\begin{proof} The translated process $X_t^{(n)} = \tilde X_t^{(n)}-k_0$ satisfies all conditions for Theorem \ref{thm:main}, with the relevant quantities (for some $\delta>0$) defined as follows (cf. \eqref{eq:sigmaa},\eqref{eq:a},\eqref{eq:d},\eqref{eq:rho}).
\begin{align}
&\sigma = \sqrt{\frac{n}{\lambda}}, && (2+\delta)\sqrt{\frac{n\log(n)}{\lambda}} \leq a_n \ll n,\\ &d_n = \lambda-1+o(1), &&
\rho_n = 1 -\frac{\lambda a_n}{2n}+ o\!\left(\frac{a_n}{n}\right)
\end{align} 
For $\kappa_n$, using the rates in \eqref{eq:contactrates}, we find
\begin{align}
\kappa_n = \max_{\frac{a_n}{2}\leq |k|\leq a_n} (\lambda_k+\mu_k) = 2k_0+\mathcal{O}(a_n)= \frac{2(\lambda-1)}{\lambda}n +o(n).
\end{align}
Now we can compute the bounds of the interval $I_n$ as defined in Theorem \ref{thm:main}. 
\begin{align}
&d_n^{-1}\log(a_n) = o(\log(n)).\\
&\kappa_n^{-1}\rho_n^{-a_n/2} =\exp\left(\frac{\lambda a_n^2}{4n}+o\left(\frac{a_n^2}{n}\right)\right).
\end{align}
Define $c_n\to 0$ by
\begin{align}
c_n = \frac{2d_n^{-1}\log(a_n)}{n} = o\left(\frac{\log(n)}{n}\right), 
\end{align}
and let $a_n = 2n\sqrt{g(n)}$ for some $g(n)$ that (slowly) goes to 0. Then, for large $n$,
\begin{align}
c_n\kappa^{-1}\rho_n^{-a_n/2}= \exp\left((\lambda+o(1))g(n)\cdot n\right)\geq \exp(g(n)\cdot n),
\end{align}
so that the interval $I_n$ contains the interval $\left[\tfrac{n}{2},\exp(g(n)\cdot n)\right]$. 

Fix $n$, write $\tilde X_t$ for $\tilde X_t^{(n)}$ and assume that 
$c\log(n)\leq \tilde X_0<k_0-\tfrac{a_n}{2}$. Let $T = \inf\{t\geq 0: k_0-\tfrac{a_n}{2}\leq \tilde X_t \leq k_0+\tfrac{a_n}{2}\}$. Note that for all $t\geq 0$ and all $k$, we can write
\begin{align}\label{eq:boundX}
\mathbb{P}(\tilde X_t\leq k) = \mathbb{P}(\tilde X_t\leq k\mid T\leq \tfrac n 2)&-\mathbb{P}(\tilde X_t\leq k\mid T\leq \tfrac n 2)\cdot\mathbb{P}(T>\tfrac n 2)\nonumber \\
&+\mathbb{P}(\tilde X_t \leq k \mid T>\tfrac n 2)\cdot\mathbb{P}(T>\tfrac n 2).
\end{align}
Let $\{\tilde Y_t\}_{t\geq 0}$ be an independent realization of the contact process on $K_n$, with initial condition $\tilde Y_0 = \lceil k_0-\tfrac{a_n}{2}\rceil$. If $T\leq \tfrac{n}{2}$, then $\tilde X_T = \lceil k_0-\tfrac{a_n}{2}\rceil$. In this case $\tilde X_{t+T}$ has the same distribution as $\tilde Y_t$, since the contact process is a Markov process. Therefore, for all $k$, all $t\geq \tfrac n 2$ and any real number $x$, 
\begin{align}\label{eq:grens1}
\left\lvert\mathbb{P}(\tilde X_t\leq k\mid T\leq \tfrac n 2)-\mathbb{P}(Z\leq x)\right\rvert &\leq \sup_{0\leq s\leq \tfrac n 2}\left\lvert\mathbb{P}(\tilde Y_{t-s}\leq k)-\mathbb{P}(Z\leq x)\right\rvert.
\end{align}
Instead of $k$, we could take any real $y$. Taking $y=\sigma_n x+k_0$ and writing $Y_t$ for $\tilde Y_t-k_0$, the probability on the right can be written
\begin{align}
\mathbb{P}(\tilde Y_{t-s}\leq y) =  \mathbb{P}\left(\frac{Y_{t-s}}{\sigma_n}\leq x\right).
\end{align}
Substitute this and take the supremum over $t\in [n,e^{g(n)\cdot n}]$:
\begin{align}
&\sup_{n\leq t \leq \exp(g(n)\cdot n)}\sup_{0\leq s\leq \tfrac n 2}\left\lvert\mathbb{P}\left(\frac{Y_{t-s}}{\sigma_n}\leq x\right)-\mathbb{P}(Z\leq x)\right\rvert\nonumber\\
&\hspace{4cm} = \sup_{\tfrac{n}{2}\leq t \leq \exp(g(n)\cdot n)}\left\lvert\mathbb{P}\left(\frac{Y_{t}}{\sigma_n}\leq x\right)-\mathbb{P}(Z\leq x)\right\rvert
\end{align}
Finally, using \eqref{eq:boundX} we obtain
\begin{align}\label{eq:boundXYZ}
&\hspace{-1cm}\sup_{t\in [n,\exp(g(n)\cdot n)]}\left\lvert\mathbb{P}(\tfrac{\tilde X_t-k_0}{\sigma_n}\leq x)-\mathbb{P}(Z\leq x)\right\rvert\nonumber\\
&\leq \mathbb{P}(T>\tfrac n 2)+\sup_{t\in [\tfrac{n}{2},\exp(g(n)\cdot n)]}\left\lvert\mathbb{P}(\tfrac{Y_t}{\sigma_n}\leq x)-\mathbb{P}(Z\leq x)\right\rvert
\end{align}
Similar arguments give the same conclusion for $\tilde X_0 > k_0+\tfrac{a_n}{2}$ (with $\tilde Y_0 = \lceil k_0+\tfrac{a_n}{2}\rceil$). For $k_0-\tfrac{a_n}{2}\leq\tilde X_0\leq k_0+\tfrac{a_n}{2}$, we can just take $\tilde Y_t = \tilde X_t$. Hence, for any initial distribution $\tilde X_0\geq c\log(n)$, there is some initial distribution $\tilde Y_0$ on $[k_0-\tfrac{a_n}{2},k_0+\tfrac{a_n}{2}]$ such that \eqref{eq:boundXYZ}
holds. By Theorem \ref{thm:main} and Lemma \ref{lem:contactprocess}, we conclude that 
\begin{align}\label{eq:conclusion}
\lim_{n\to\infty}\sup_{t\in [n,\exp(g(n)\cdot n)]}\left\lvert\mathbb{P}(\tfrac{\tilde X_t-k_0}{\sigma_n}\leq x)-\mathbb{P}(Z\leq x)\right\rvert =0.
\end{align}
Clearly it is also sufficient if $\lim_{n\to\infty}\mathbb{P}(\tilde X_0\geq c\log(n)) = 1$, since $\mathbb{P}(X_t^{(n)}\leq x)$ converges to $\mathbb{P}(X_t^{(n)}\leq x\mid \tilde X_0\geq c\log(n))$, and \eqref{eq:conclusion} applies to the conditioned version of $X_t^{(n)}$. As a final remark, note that we can replace $\tilde X_t^{(n)}-k_0$ in the conclusion by  $\tilde X_t^{(n)}-\mu$ since $\tfrac{k_0-\mu}{\sigma_n}\to 0$.
\end{proof}

\appendix
\section{Proof of a technical lemma}

The following lemma is used in the proof of Proposition \ref{prop:statdist}.

\begin{lemma}\label{lem:sumnormdens}
For $b\geq \sigma\sqrt{2\log(\sigma)}$ we have
	\[ \sqrt{2\pi}\sigma -(1+2\sqrt{2\pi}) \leq \sum_{k=-b}^{b} e^{-\frac12\frac{k^2}{\sigma^2}} \leq \sqrt{2\pi}\sigma+1.\]	
 In particular, this implies for $\sigma$ large enough
 \[ \sqrt{2\pi}\sigma e^{-b/\sigma^2} \leq \sum_{k=-b}^{b} e^{-\frac12\frac{k^2}{\sigma^2}} \leq \sqrt{2\pi}\sigma e^{b/\sigma^2}.\]
 \end{lemma}

\begin{proof} Upper bound (for all $b\in\mathbb{N}$ and all $\sigma>0$):
\begin{align}
\sum_{k=-b}^be^{-\frac12\frac{k^2}{\sigma^2}} & = 1+2\sum_{k=1}^be^{-\frac12\frac{k^2}{\sigma^2}}\leq 1+2\int_0^b e^{-\frac12\frac{s^2}{\sigma^2}}ds\nonumber\\
&= 1+\sigma\int_{-b/\sigma}^{b/\sigma}e^{-\frac12 s^2}ds \leq 1+\sigma\int_{-\infty}^\infty e^{-\frac12 s^2}ds = 1+\sigma\sqrt{2\pi}\nonumber\\
& = \sigma\sqrt{2\pi}(1+\frac{1}{\sigma\sqrt{2\pi}})\leq \sigma\sqrt{2\pi}e^{1/\sigma\sqrt{2\pi}} \leq \sigma\sqrt{2\pi}e^{b/\sigma^2}.
\end{align}
Lower bound (using a Chernoff bound for the tail of a standard Gaussian $Z$):
\begin{align}
\sum_{k=-b}^be^{-\frac12\frac{k^2}{\sigma^2}} & = 1+2\sum_{k=1}^be^{-\frac12\frac{k^2}{\sigma^2}}\geq 1+2\int_1^{b} e^{-\frac12\frac{s^2}{\sigma^2}}ds\nonumber\\
&\geq -1+\int_{-b}^{b} e^{-\frac12\frac{s^2}{\sigma^2}}ds\geq -1+\sigma\sqrt{2\pi}\cdot \mathbb{P}\left(|Z|\leq \sqrt{2\log(\sigma)}\right)\nonumber\\
&\geq -1+\sigma\sqrt{2\pi}\cdot(1-2e^{-\frac{\sqrt{2\log(\sigma)}^2}{2}})= -1-2\sqrt{2\pi}+\sigma\sqrt{2\pi}\nonumber\\
& = \sigma\sqrt{2\pi}(1-\frac{1+2\sqrt{2\pi}}{\sigma\sqrt{2\pi}}) \geq \sigma\sqrt{2\pi} e^{-3/\sigma} \geq \sigma\sqrt{2\pi} e^{-b/\sigma^2}, 
\end{align}
for $\sigma$ large enough.
\end{proof}

\bibliographystyle{plain}
\bibliography{contactprocess}

\end{document}